\documentclass{amsart}

\usepackage{mathrsfs}  
\usepackage{amssymb,latexsym,fullpage}
\usepackage{mathtools}
\usepackage{enumerate}  
\usepackage{a4wide}    

\usepackage[colorlinks=true,citecolor=blue,linkcolor=blue]{hyperref}
\usepackage{cleveref}
\usepackage{latexsym,amsmath,amssymb,comment}
\usepackage{accents}
\usepackage{color}  

\definecolor{bittersweet}{rgb}{1.0, 0.44, 0.37}

\newcommand{\chgdarm}[1]{{ #1}}
\newcommand{\chgdt}[1]{{ #1}}

\title[Bourgain--Brezis--Mironescu Convergence result and Triebel-Lizorkin Spaces]{Bourgain--Brezis--Mironescu Convergence via Triebel-Lizorkin Spaces}
 \author{Denis Brazke}
 \address[Denis Brazke]{Department of Mathematics,
 University of Heidelberg,
 Im Neuenheimer Feld 205,
 69120 Heidelberg, Germany}
 \email{denis.brazke@uni-heidelberg.de}
 \author[Armin Schikorra]{Armin Schikorra}
 \address[Armin Schikorra]{Department of Mathematics,
 University of Pittsburgh,
 301 Thackeray Hall,
 Pittsburgh, PA 15260, USA}
 \email{armin@pitt.edu}
 \author[Po-Lam Yung]{Po-Lam Yung}
 \address[Po-Lam Yung]{Mathematical Sciences Institute, The Australian National University, Canberra, Australia}
 \email{polam.yung@anu.edu.au}

plus 6pt minus 12pt
\def\eps{\varepsilon}

\def\N{{\mathbb N}}

\def\S{{\mathbb S}}

\newtheorem{theorem}{Theorem}

\newtheorem{lemma}[theorem]{Lemma}
\newtheorem{corollary}[theorem]{Corollary}
\newtheorem{proposition}[theorem]{Proposition}

\theoremstyle{definition}
\newtheorem{remark}[theorem]{Remark}
\newtheorem{definition}[theorem]{Definition}

\newtheorem{question}[theorem]{Question}

\def\dist{{\rm dist\,}}

\def\lip{{\rm Lip\,}}

\def\supp{{\rm supp\,}}

\newcommand{\R}{\mathbb{R}}

\newcommand{\Z}{\mathbb{Z}}

\newcommand{\brac}[1]{\left (#1 \right )}
\newcommand{\abs}[1]{\left |#1 \right |}

\newcommand{\Sw}{\mathscr{S}}

\newcommand{\barint}{
\rule[.036in]{.12in}{.009in}\kern-.16in \displaystyle\int }

\newcommand{\barcal}{\mbox{$ \rule[.036in]{.11in}{.007in}\kern-.128in\int $}}

\def\mvint_#1{\mathchoice
          {\mathop{\vrule width 6pt height 3 pt depth -2.5pt
                  \kern -8pt \intop}\nolimits_{\kern -3pt #1}}%
          {\mathop{\vrule width 5pt height 3 pt depth -2.6pt
                  \kern -6pt \intop}\nolimits_{#1}}%
          {\mathop{\vrule width 5pt height 3 pt depth -2.6pt
                  \kern -6pt \intop}\nolimits_{#1}}%
          {\mathop{\vrule width 5pt height 3 pt depth -2.6pt
                  \kern -6pt \intop}\nolimits_{#1}}}

\numberwithin{theorem}{section} \numberwithin{equation}{section}

\newcommand{\lap}{\Delta }

\newcommand{\aleq}{\lesssim}
\newcommand{\ageq}{\gtrsim}
\newcommand{\aeq}{\approx}
\newcommand{\Rz}{\mathcal{R}}
\newcommand{\laps}[1]{(-\lap) ^{\frac{#1}{2}}}

\begin{document}
\begin{abstract}
We study a convergence result of Bourgain--Brezis--Mironescu (BBM) using Triebel-Lizorkin spaces. It is well known that as spaces $W^{s,p} = F^{s}_{p,p}$, and $H^{1,p} = F^{1}_{p,2}$. When $s\to 1$, the $F^{s}_{p,p}$ norm becomes the $F^{1}_{p,p}$ norm but BBM showed that the $W^{s,p}$ norm becomes the $H^{1,p} = F^{1}_{p,2}$ norm. Naively, for $p \neq 2$ this seems like a contradiction, but we resolve this by providing embeddings of $W^{s,p}$ into $F^{s}_{p,q}$ for $q \in \{p,2\}$ with sharp constants with respect to $s \in (0,1)$. As a consequence we obtain an $\mathbb{R}^N$-version of the BBM-result, and obtain several more embedding and convergence theorems of BBM-type that to the best of our knowledge are unknown. 
\end{abstract}

\maketitle

\section{Introduction and Main Results}

\subsection{{Previous results}}
For $s \in (0,1)$, $p \in (1,\infty)$ and an open set $\Omega \subset \R^N$ the ${\dot{W}}^{s,p}$-Gagliardo-seminorm is defined as
\[
 [f]_{\dot W^{s,p}(\Omega)} = \brac{\int_{\Omega}\int_{\Omega} \frac{|f(x)-f(y)|^p}{|x-y|^{N+sp}}\, dx\, dy}^{\frac{1}{p}} \equiv \left \|\frac{f(x)-f(y)}{|x-y|^{\frac{N}{p}+s}}\right \|_{L^p(\Omega \times \Omega)}.
\]
For $s = 1$ we denote the usual ${\dot{H}}^{1,p}$-Sobolev space seminorm by
\[
 [f]_{\dot H^{1,p}(\Omega)} = \|\nabla f\|_{L^p(\Omega)}
\]
{and write $H^{1,p}$ for the inhomogeneous Sobolev space so that
\[
H^{1,p}(\Omega) := \{f \in L^p(\Omega) : \nabla f \in L^p(\Omega)\}.
\]}

In the influential paper \cite{BBM01} Bourgain--Brezis--Mironescu {showed} 
that for {any smooth bounded domain $\Omega \subset \R^N$ and} any $f \in H^{1,p}(\Omega)$ we have
\begin{equation}\label{eq:BBMglobaleq}
 \tag{BBM1}  \|\nabla f\|_{L^p(\Omega)}  = \left( \frac{p}{k(p,N)} \right)^{1/p} \lim_{s \to 1^-} (1-s)^{\frac{1}{p}} [f]_{\dot W^{s,p}(\Omega)},
\end{equation}
where $k(p,N) := \int_{\S^{N-1}} |e \cdot \omega|^p d\omega$ and $e$ is any unit vector in $\R^N$. 
Even more crucially, Bourgain--Brezis--Mironescu {established} the following convergence result.
\begin{theorem}[Bourgain--Brezis--Mironescu \cite{BBM01}]\label{th:bbm}
Let $\Omega \subset \R^N$ be open and bounded with smooth boundary, and $p \in (1,\infty)$. 
\begin{itemize}
\item[(BBM2)] Assume that $f_k \in C_c^\infty(\Omega)$ such that 
\[
 f_k \rightharpoonup f \quad \text{weakly in $L^p(\Omega)$ as $k \to \infty$}.
\]
Let $(s_k)_{k \in \N} \subset (0,1)$ such that $s_k \uparrow 1$ and assume that
\[
\Lambda := \sup_{k} \brac{\|f_k\|_{L^p(\Omega)} + (1-s_k)^{\frac{1}{p}} [f_k]_{\dot W^{s_k,p}(\Omega)}} < \infty.
\]
Then $f \in H^{1,p}(\Omega)$ and we have 
\[
 \|f\|_{L^p(\Omega)} + \|\nabla f\|_{L^p(\Omega)} \leq C\, 
 \Lambda.
\]
The constant $C$ depends only $p$ and $N$. Also, $f_k \xrightarrow{k \to \infty} f$ strongly in $L^p_{loc}(\Omega)$.
\end{itemize}
\end{theorem}

See also \cite{BPS16,Ponce04,dTGCV20} for related results, \cite{Milman1,Milman2} for an interpretation via interpolation space, and \cite{Mazya1,Mazya2} for the regime $s \to 0$.

\subsection{{Questions on $\R^N$}}

{In this paper, we explore what happens when the bounded domain $\Omega$ above is replaced by the whole space $\R^N$. It is relatively easy to show that \eqref{eq:BBMglobaleq} holds with $\Omega$ replaced by $\R^N$; we provide a short proof in \Cref{s:bbmglobaleq}. Our main result will be an analog of \Cref{th:bbm} on $\R^N$. In fact,}
from the point of view of Harmonic Analysis, \Cref{th:bbm} seems like a surprising result, as we shall explain here. Denote the {homogeneous} Triebel-Lizorkin norm $[\, \cdot \,]_{\dot{F}^{s}_{p,p}(\R^N)}$ by
\[
 [f]_{\dot{F}^{s}_{p,p}(\R^N)} = \brac{\int_{\R^N} \sum_{j \in \Z} 2^{sjp}|\Delta_j f(x)|^p \ dx}^{\frac{1}{p}}.
\]
Here $\Delta_j f$ are the Littlewood-Paley projections {(see \Cref{s:triebelprelim} for their definitions).}
It is well-known that for $s \in (0,1)$, $p \in (1,\infty)$,
\[
 [f]_{\dot{F}^{s}_{p,p}(\R^N)} \aeq [f]_{\dot W^{s,p}(\R^N)},
\]
whenever $f \in \Sw(\R^N)$, where $\Sw(\R^N)$ denotes the set of Schwartz functions on $\R^N$.
However, since $\|f\|_{L^p(\R^N)} \aeq \|f\|_{\dot{F}^{0}_{p,2}}$ we have 
\[
 [f]_{\dot{F}^{1}_{p,2}(\R^N)} \aeq \|\nabla f\|_{L^p(\R^N)}.
\]
From the definition of Triebel-Lizorkin spaces, it easily follows (cf. \Cref{la:triebelconv})
\[
 \lim_{s \to 1} [f]_{\dot{F}^{s}_{p,p}(\R^N)}=[f]_{\dot{F}^{1}_{p,p}(\R^N)}.
\]
So {if \Cref{th:bbm} holds true on $\R^N$, it then}
seems to suggest that in some way $\dot{W}^{s,p} \aeq_{s,p} \dot{F}^{s}_{p,p}$ ``converges to'' $\dot{H}^{1,p} \aeq_{p} \dot{F}^{1}_{p,2}$, which appears \chgdarm{to be} a contradiction to the above, because for $p \neq 2$ we have that $\dot{F}^{1}_{p,2} \not \aeq \dot{F}^{1}_{p,p}$. These statements, of course, do not make any sense, because spaces do not converge, but norms. The aim of this note is to clarify the effects we are seeing here{{, which we achieve by clarifying various relationships between the $\dot{W}^{s,p}$, $\dot{F}^s_{p,p}$ and $\dot{F}^s_{p,2}$ seminorms for $0 < s < 1$}.

\subsection{{Results about $\dot{F}^s_{p,p}$}}
Our first main theorem is the following {quantitative comparison between the $\dot{W}^{s,p}$ and the $\dot{F}^s_{p,p}$ seminorms}. 

\begin{theorem}\label{th:mainFspp}
Let $N \ge 1$, $p \in (1,\infty)$. Then there exists $C = C(N,p) > 0$, such that for every $s \in (0,1)$ and $f \in \Sw(\R^N)$,
\begin{enumerate}
 \item if $1 < p \leq 2$:
\begin{equation}\label{eq:Fspppleq2}
 C^{-1} \brac{\frac{1}{s^{\frac{1}{2}}} + \frac{1}{(1-s)^{\frac{1}{2}}}}\, [f]_{\dot{F}^{s}_{p,p}(\R^N)} \leq [f]_{\dot W^{s,p}(\R^N)} \leq C \brac{\frac{1}{s^{\frac{1}{p}}} + \frac{1}{(1-s)^{\frac{1}{p}}}}\, [f]_{\dot{F}^{s}_{p,p}(\R^N)}.
\end{equation}
\item if $2 \leq p < \infty$:
\begin{equation}\label{eq:Fspppgeq2}
 C^{-1} \brac{\frac{1}{s^{\frac{1}{p}}} + \frac{1}{(1-s)^{\frac{1}{p}}}}\, [f]_{\dot{F}^{s}_{p,p}(\R^N)} \leq [f]_{\dot W^{s,p}(\R^N)} \leq C \brac{\frac{1}{s^{\frac{1}{2}}} + \frac{1}{(1-s)^{\frac{1}{2}}}}\, [f]_{\dot{F}^{s}_{p,p}(\R^N)}.
\end{equation}
\end{enumerate}
\end{theorem}
The upper bounds in \eqref{eq:Fspppleq2} and \eqref{eq:Fspppgeq2} have been proven by Gu and the third author in \cite{GY21}.

As an immediate corollary we obtain the following Sobolev-type inequality for $p=2$. It is well-known and elementary to show that
\[
[f]_{\dot W^{s,2}(\R^N)} \leq C_{s,t} \brac{\|f\|_{L^2(\R^N)} + [f]_{\dot W^{t,2}(\R^N)}}, \quad  \text{for $0 < s \leq t < 1$}.
\]
The main nontriviality in the corollary below is the prefactor $\min\{s,(1-s)\}^{\frac{1}{2}}$ on the left-hand side and $\min\{t,(1-t)\}^{\frac{1}{2}}$ on the right-hand side. We do not know if a similar statement is true for any $p \in (1,\infty)$, see \Cref{q:sharpsob}.
\begin{corollary}\label{co:sharpshob2}
{Let $N \geq 1$. Then there exists $C = C(N) > 0$, such that for all }
$0 < s \leq t <1$ and $f \in \Sw(\R^N)$,
\[
 \min\{s,(1-s)\}^{\frac{1}{2}} [f]_{\dot W^{s,2}(\R^N)} \leq C \brac{
 \|f\|_{L^2(\R^N)} + \min\{t,(1-t)\}^{\frac{1}{2}} [f]_{\dot W^{t,2}(\R^N)}}.
\]
\end{corollary}
For the convenience of the reader, we give the details of the proof \Cref{co:sharpshob2} in \Cref{s:co:sharpshob2}.

\begin{remark}[Sharpness of the constants]
To some extent the constants in \Cref{th:mainFspp} are sharp, as can be shown using the results of \cite{BBM01}.
\begin{enumerate}
\item Observe that in general for $p < 2$
\[
\brac{\frac{1}{s^{\frac{1}{p}}} + \frac{1}{(1-s)^{\frac{1}{p}}}}\, [f]_{\dot{F}^{s}_{p,p}(\R^N)} \not \leq C [f]_{\dot W^{s,p}(\R^N)} 
\]
for $C = C(N,p) > 0$. Indeed, if that was true for all $s \in (0,1)$, we could pick a function $f \in H^{1,p}(\R^N)$ with compact support that does not belong to $F^{1}_{p,p}(\R^N)$. From \cite{BBM01} we would then obtain that 
\[
 \limsup_{s \to 1^-} \ (1-s)^{\frac{1}{p}} [f]_{\dot W^{s,p}(\R^N)} < \infty,
\]
however we have 
\[
 \lim_{s \to 1^-} [f]_{\dot{F}^{s}_{p,p}(\R^N)} = [f]_{\dot{F}^{1}_{p,p}(\R^N)} = \infty.
\]
\item Similarly, for $p > 2$ in general
\[
 [f]_{{\dot{W}}^{s,p}(\R^N)} \not \leq C \brac{\frac{1}{s^{\frac{1}{p}}} + \frac{1}{(1-s)^{\frac{1}{p}}}}\, [f]_{\dot{F}^{s}_{p,p}(\R^N)} \\
\]
for $C = C(N,p) > 0$. To obtain a counterexample in this case take, $f \in F^{1}_{p,p}(\R^N)$ with compact support and $f \not \in H^{1,p}(\R^N)$. Then, again by \cite{BBM01}
\[
 \limsup_{s \to 1^-} \ (1-s)^{\frac{1}{p}} [f]_{{\dot{W}}^{s,p}(\R^N)}  = \infty,
\]
however 
\[
 \liminf_{s \to 1^-} \ [f]_{\dot{F}^{s}_{p,p}(\R^N)} = [f]_{\dot{F}^{1}_{p,p}(\R^N)} < \infty.
\]
\end{enumerate}
\end{remark}

\subsection{{Results about $\dot{F}^s_{p,2}$}}

{Next we explore relationships between the $\dot{W}^{s,p}$ and the $\dot{F}^s_{p,2}$ seminorms}. } Observe that while \Cref{th:mainFspp} is a nice characterization, and we obtain some convergence for functions with uniformly bounded $(1-s)^{\frac{1}{p}}[f]_{\dot W^{s,p}(\R^N)}$-norms, we do not recover \Cref{th:bbm} yet. For this we need a different space. Namely, we obtain the following $\dot{F}^{s}_{p,2}$-estimate and the main focus should be on how changing from $\dot{F}^{s}_{p,p}$ to $\dot{F}^{s}_{p,2}$ improves the dependency on $s$ and $(1-s)$.
\begin{theorem}\label{th:mainFsp2}
Let $N \ge 1$, $p \in (1,\infty)$. Then there exists $C = C(N,p) > 0$, such that for all $s \in (0,1)$ and $f \in \Sw(\R^N)$, 
\begin{enumerate}
\item if $1 < p \leq 2$: 
\begin{equation}\label{eq:mainFsp2:1}
 C^{-1} \brac{\frac{1}{s^{\frac{1}{p}}} + \frac{1}{(1-s)^{\frac{1}{p}}}}\, [f]_{\dot{F}^{s}_{p,2}(\R^N)} \leq [f]_{\dot W^{s,p}(\R^N)} .
\end{equation}
\item if $2 \leq p  < \infty$:
\begin{equation}\label{eq:mainFsp2:2}
 [f]_{\dot W^{s,p}(\R^N)}  \leq C \brac{\frac{1}{s^{\frac{1}{p}}} + \frac{1}{(1-s)^{\frac{1}{p}}}}\, [f]_{\dot{F}^{s}_{p,2}(\R^N)}.
\end{equation}
\end{enumerate}
\end{theorem}

{The upper bound for $[f]_{\dot{F}^s_{p,2}(\R^N)}$ in \eqref{eq:mainFsp2:1} in }\Cref{th:mainFsp2} provides a full, $\R^N$-version of \Cref{th:bbm} if $p \leq 2$, see \Cref{co:strongconv} below. For $p \geq 2$
{the desired upper bound for $[f]_{\dot{F}^s_{p,2}(\R^N)}$ will be provided by the following Sobolev-type estimate: see \eqref{eq:sobolev2}.}
\begin{theorem}[Sobolev-Estimate]\label{th:sobolev}
Let {$N \geq 1$,} $p \in (1,\infty)$.
\begin{enumerate}
\item {Then there exists $C = C(N,p) > 0$, such that for $0 \leq r < s < t \leq 1$ and $f \in \Sw(\R^N)$,}
\begin{equation}\label{eq:sobolev1}
 [f]_{\dot W^{s,p}(\R^N)}  \leq C \left(\frac{1}{(s-r)^{\frac{1}{p}}}\, [f]_{\dot{F}^{r}_{p,2}} + \frac{1}{(t-s)^{\frac{1}{p}}}\, [f]_{\dot{F}^{t}_{p,2}} \right). 
\end{equation}
\item Let $\Lambda > 1$. {Then there exists $C = C(N,p,\Lambda) > 0$, such that the following holds: Let
\chgdt{$s \in [1-\frac{1}{2\Lambda}, 1)$}%
.} Let $\bar{r} \in (0,s)$ such that $(1-\bar{r})=\Lambda (1-s)$. Pick $r \in [0,\bar{r}]$. {Then for any $f \in \Sw(\R^N)$,}
\begin{equation}\label{eq:sobolev2}
 [f]_{\dot{F}^{r}_{p,2}(\R^N)} \leq C  \brac{\|f\|_{L^p(\R^N)} + (1-s)^{\frac{1}{p}} [f]_{\dot W^{s,p}(\R^N)}}.
\end{equation}
\end{enumerate}
\end{theorem}
Applying \cite[Theorem 1]{BBM01} to $\rho(x) = |x|^{-N-(1-s)p}$ one obtains for a bounded set $\Omega$
\[
 \sup_{s \in (0,1)} (1-s)^{\frac{1}{p}} [f]_{\dot W^{s,p}(\Omega)} \leq C(N,\Omega,p) \|\nabla f\|_{L^p(\Omega)}.
\]
As an immediate corollary of \Cref{th:sobolev}, we find a variant of this inequality on $\R^N$ and even obtain a fractional version of it. In the following we denote {for non-integral values of $s$}
\[
 \dot{H}^{s,p}(\R^N)\equiv \dot{F}^{s}_{p,2}(\R^N),
\]
whose seminorm $[f]_{\dot{H}^{s,p}(\R^N)} = \|\laps{s} f\|_{L^p(\R^N)}$ is defined via the fractional Laplacian. Observe that $\|\laps{1} f\|_{L^p(\R^N)} \aeq \|\nabla f\|_{L^p(\R^N)}$ for any $p \in (1,\infty)$ by the $L^p$-boundedness of the Riesz transforms{, so for our purposes it does not really matter whether we defined $[f]_{\dot{H}^{1,p}}$ to be $\|\nabla f\|_{L^p}$ or $\|(-\Delta)^{1/2} f\|_{L^p}$}. From \eqref{eq:sobolev1} it is easy to deduce

\begin{corollary}\label{co:fracBBMupper}
Let $N \ge 1$, $p \in (1,\infty)$, $0 < \theta < 1$. Then there exists $C = C(N,p,\theta) > 0$, such that for $s \in (\theta, 1]$ and $f \in \Sw(\R^N)$,
\begin{equation}\label{eq:co:fracBBMupper:s}
 \sup_{r \in [\theta,s)} (s-r)^{\frac{1}{p}} [f]_{\dot W^{r,p}(\R^N)} \leq C\brac{\|f\|_{L^p(\R^N)} + \|\laps{s} f\|_{L^p(\R^N)}}.
\end{equation}
In particular, setting $s = 1$, we obtain
\[
 \sup_{r \in [\theta,1)} (1-r)^{\frac{1}{p}} [f]_{\dot W^{r,p}(\R^N)} \leq C \brac{\|f\|_{L^p(\R^N)} + \|\nabla f\|_{L^p(\R^N)}}.
\]
\end{corollary}

\begin{remark}
Barring the independence of the constant on $s$, the case $s < 1$ in \eqref{eq:co:fracBBMupper:s} is only interesting for the case $p < 2$. For $p \geq 2$ and $s <1$ it is an obvious (and non-optimal) estimate. Indeed, if $f \in \Sw(\R^N)$, $p \geq 2$ and $r \in (0,1)$, we have 
\[
 [f]_{\dot W^{r,p}(\R^N)} \aeq [f]_{\dot{F}^{r}_{p,p}} \leq [f]_{\dot{F}^{r}_{p,2}} \aeq \|\laps{r} f\|_{L^p(\R^N)},
\]
where the implicit constants depend on $r, p$ and $N$. If $r$ is in a compact subinterval of $(0,1)$, then the constants can be taken independent of $r$. 
Since
\[
 (s-r)^{\frac{1}{p}} [f]_{\dot W^{r,p}(\R^N)}  \leq C [f]_{\dot W^{r,p}(\R^N)} 
\]
and
\[
\|\laps{r} f\|_{L^p(\R^N)} \leq C \brac{ \|f\|_{L^p(\R^N)} + \|\laps{s} f\|_{L^p(\R^N)}}
\]
for $r \in [0,s]$ (with constants independent of $r$ and $s$), if $s \leq 1-\theta'$ for some $\theta' > 0$, then {whenever $\theta > 0$}
\[
\sup_{r \in [\theta,s)} (s-r)^{\frac{1}{p}} [f]_{\dot W^{r,p}(\R^N)} \leq C \brac{\|f\|_{L^p(\R^N)} + \|\laps{s} f\|_{L^p(\R^N)}}
\]
with a constant depending on $\theta, \theta', p$ and $N$ but not on $r$ and $s$.
\end{remark}

\subsection{{Back to Bourgain--Brezis--Mironescu's convergence result}}

From \eqref{eq:sobolev2} and Rellich--Kondrachov theorem we recover in particular (BBM2) of \Cref{th:bbm} -- actually with a stronger convergence than is commonly considered in the literature.
\begin{corollary}\label{co:strongconv}
Let $p \in (1,\infty)$, assume that $f_k \in \Sw(\R^N)$ such that 
\[
 f_k \rightharpoonup f \quad \text{weakly in $L^p(\R^N)$ as $k \to \infty$}.
\]
Let $(s_k)_{k \in \N} \subset (0,1)$ such that $s_k \uparrow 1$ and assume that
\begin{equation}\label{eq:co:strongconv435}
\Lambda := \sup_{k} \brac{\|f_k\|_{L^p(\R^N)} + (1-s_k)^{\frac{1}{p}} [f_k]_{\dot W^{s_k,p}(\R^N)}} < \infty.
\end{equation}
Then $f \in H^{1,p}(\R^N)$ and we have 
\[
 \|f\|_{L^p(\R^N)} + \|\nabla f\|_{L^p(\R^N)} \leq C\, 
 \Lambda.
\]
The constant $C$ depends on $p$ and $N$.

Also, $f_k \xrightarrow{k \to \infty} f$ strongly in $H^{t,p}_{loc}(\R^N)$ for any $t \in [0,1)$, that is 
\begin{equation}\label{eq:co:strongconv4386}
 \lim_{k \to \infty} \|\laps{t} f_k - \laps{t} f\|_{L^p(K)} = 0 \quad \forall \text{compact sets $K \subset \R^N$},
\end{equation}
and for any $t \in (0,1)$
\begin{equation}\label{eq:co:strongconv43862}
 \lim_{k \to \infty} [f_k - f]_{\dot W^{t,p}(K)} = 0 \quad \forall \text{compact sets $K \subset \R^N$}.
\end{equation}
\end{corollary}
We give the details of the proof in \Cref{s:co:strongconv}. The above strong convergence may not be global in $\R^N$ (even strong convergence in $L^p(\R^N)$ may be false). A counterexample is given by a standard counterexample to the global Rellich-Kondrachov Theorem for $W^{1,p}(\R^N)$ (which shows that $W^{1,p}(\R^N)$ does not embed compactly into $L^p(\R^N)$): for instance, if $f \in C^{\infty}_c(\R^N)$ and $\{f_k\}_k$ is a sequence of translates of $f$ that escapes off to infinity, then $f_k$ converges weakly to $0$ in $W^{1,p}(\R^N)$, \eqref{eq:co:strongconv435} is satisfied by \Cref{co:fracBBMupper}, but $f_k$ does not converge strongly in $L^p(\R^N)$.

\subsection{Open questions and further directions}

\begin{question}\label{q:sharpsob}
Let $p \in (1,\infty)$, $0 < \theta < s<t <1$ and $f \in \Sw(\R^N)$. Is it true that there exists $C = C(N,p,\theta) > 0$, such that
\[
 \min\{s,(1-s)\}^{\frac{1}{p}} [f]_{\dot W^{s,p}(\R^N)} \leq C
 \brac{\|f\|_{L^p(\R^N)} + \min\{t,(1-t)\}^{\frac{1}{p}} [f]_{\dot W^{t,p}(\R^N)}}?
\]
\end{question}

An indication that the above might be true, is the case $p=2$, \Cref{co:sharpshob2}. Also, of course, \Cref{q:sharpsob} holds asymptotically for $s=t$ and -- in view of \cite{BBM01} -- it holds if we first let $t \uparrow 1$ and \emph{then} take $s \uparrow 1$.

Let us remark that a very rough toy-case for \Cref{q:sharpsob} are characteristic functions -- and indeed the inequality from \Cref{q:sharpsob} holds in that case: for $A \subset \R^N$ measurable we have $|\chi_A(x)-\chi_A(y)|^p = |\chi_A(x)-\chi_A(y)|^2$, and $|x-y|^{N+sp} = |x-y|^{N+\frac{sp}{2}2}$. Thus,
\[
 [\chi_A]_{\dot W^{s,p}(\R^N)} = [\chi_A]_{\dot W^{\frac{sp}{2},2}(\R^N)}^{\frac{2}{p}},
\]
so
\[
\begin{split}
 &\min\{s,(1-s)\}^{\frac{1}{p}} [\chi_A]_{\dot W^{s,p}(\R^N)} 
 =  \min\{s,(1-s)\}^{\frac{1}{p}} [\chi_A]_{\dot W^{\frac{sp}{2},2}(\R^N)}^{\frac{2}{p}}
 =  \brac{\min\{s,(1-s)\}^{\frac{1}{2}} [\chi_A]_{\dot W^{\frac{sp}{2},2}(\R^N)} }^{\frac{2}{p}}\\ 
 &\quad \aleq_p  \|\chi_A\|_{L^2}^{\frac{2}{p}} + \brac{\min\{t,(1-t)\}^{\frac{1}{2}} [\chi_A]_{\dot W^{\frac{tp}{2},2}(\R^N)} }^{\frac{2}{p}}
 = \|\chi_A\|_{L^p} +  \min\{t,(1-t)\}^{\frac{1}{p}} [\chi_A]_{\dot W^{t,p}(\R^N)}.
 \end{split}
\]

Moving on to the next question, the estimate \eqref{eq:co:fracBBMupper:s} hints towards the possibility that there might be a Brezis--Bourgain--Mironescu-type result for $s  <1$, namely it establishes {an} $H^{s,p}$-type (BBM1)-estimate. It is unclear to us if the convergence result is also true.

\begin{question}\label{q:fracBBM}
Let $p \in (1,\infty)$, assume that $f_k \in \Sw(\R^N)$ such that 
\[
 f_k \rightharpoonup f \quad \text{weakly in $L^p(\R^N)$ as $k \to \infty$}.
\]
Let $t \in (0,1)$ and  $(s_k)_{k \in \N} \subset (0,t)$ such that $s_k \uparrow t$ and assume that
\[
\Lambda := \sup_{k} \brac{\|f_k\|_{L^p(\R^N)} + (t-s_k)^{\frac{1}{p}} [f_k]_{\dot W^{s_k,p}(\R^N)}} < \infty.
\]
Is it true that $f \in H^{t,p}(\R^N)$ and that there exists $C = C(N,p,t) > 0$, such that
\[
 \lim_{\tilde{t} \uparrow t} \limsup_{k \to \infty} \|\laps{\tilde{t}} f_k\|_{L^{p}(\R^N)} \leq C \Lambda?
\]
\end{question}

Our next question concerns an extension to other Triebel-Lizorkin spaces. It is known that for $p > \frac{Nq}{N+sq}$
\begin{equation}\label{eq:wspq}
        [f]_{\dot{W}^{s,p}_q(\R^N)} {:=} \left (\int_{\R^N} \left (\int_{\R^N} \frac{|f(x)-f(y)|^q}{|x-y|^{N+sq}} dy\right )^{\frac{p}{q}}dx\right )^{\frac{1}{p}} \aeq [f]_{\dot{F}^{s}_{p,q}(\R^N)},
       \end{equation}
see~\cite{Stein1961,AMV12} for $q=2$, \cite[Section 2.5.10]{T83} for $s \geq \frac{N}{\min\{p,q\}}$ and \cite{P19} for the general $p > \frac{Nq}{N+sq}$. Unless $q = 2$, which was treated in \cite{F70}, the case of equality $p = \frac{Nq}{N+sq}$ seems to be open. 
\begin{question}\label{q:fspqversion}
What is the dependency on $s$ and $(1-s)$ as $s \downarrow 0$ or $s \uparrow 1$ in the equivalence \eqref{eq:wspq}?
\end{question}
\Cref{q:fspqversion} is related to works by Spector--Leoni, \cite{LS11,LS11b}. Of course, the limit cases $p=1$ and $p=\infty$ would also be an interesting direction, cf. \cite{D2002}.

Lastly let us mention that Bourgain--Brezis--Mironescu \cite{BBM01} (see also \cite{Ponce04}) the singular kernel $|x-y|^{-N-sp}$ is only one special case considered. In general they work with family of kernels $\rho_n$ that suitably approximate $|x-y|^{-p} \delta_{x,y}$. It might be possible to adapt our methods to treat this case as well, in the sense that as $n \to \infty$ the corresponding $\rho_n$-seminorm controls more and more frequencies estimated in ${\dot{F}}^{s}_{p,2}$.

{The paper will be organized as follows. In \Cref{sect:prelim} we collect a few basic results about Triebel-Lizorkin spaces. In \Cref{sect:p=2} we give a simple proof of \Cref{th:mainFspp} in the special case $p = 2$. In \Cref{sect:upper1.21.5} we prove the upper bounds for $[f]_{\dot{W}^{s,p}}$ in \Cref{th:mainFspp} and \Cref{th:mainFsp2}, i.e. the second inequalities of \eqref{eq:Fspppleq2} and \eqref{eq:Fspppgeq2}, and the inequality \eqref{eq:mainFsp2:2}. In \Cref{s:sobolevbound} we prove the upper bound \eqref{eq:sobolev1} for $[f]_{\dot{W}^{s,p}}$ in \Cref{th:sobolev}. In \Cref{sect:lower} we prove the lower bounds for $[f]_{\dot{W}^{s,p}}$ in \Cref{th:mainFspp}, \Cref{th:mainFsp2} and \Cref{th:sobolev}, i.e. the first inequalities of \eqref{eq:Fspppleq2} and \eqref{eq:Fspppgeq2}, and the inequalities \eqref{eq:mainFsp2:1} and \eqref{eq:sobolev2}. In \Cref{s:co:strongconv} we prove \Cref{co:strongconv}. Finally, in \Cref{s:bbmglobaleq} we prove \eqref{eq:BBMglobaleq} with $\Omega$ replaced by $\R^N$, and in \Cref{s:co:sharpshob2} we give a short proof of \Cref{co:sharpshob2}. }

{\subsection*{Recent progress in \cite{DM21}} After finishing this manuscript, Dom\'inguez and Milman \cite{DM21} settled \Cref{q:sharpsob} and \Cref{q:fracBBM}, using heavy interpolation machinery. They also provide alternative proofs of our main theorems via these interpolation and extrapolation techniques.}
\subsection*{Acknowledgments}
D.B. is funded by the Deutsche Forschungsgemeinschaft (DFG, German Research Foundation) under Germany's Excellence Strategy EXC 2181/1 - 390900948 (the Heidelberg STRUCTURES Excellence Cluster). 
A.S. is funded by the NSF Career award DMS-2044898 and Simons foundation grant no 579261. P-L. Y. is funded by an Australian Future Fellowship FT200100399.
Discussions with James Scott and Andreas Seeger are gratefully acknowledged.

\section{Preliminaries} \label{sect:prelim}
In this section we gather preliminary results that most likely are all widely known. Throughout the paper we use the notation $A \aleq B$ whenever there is a multiplicative constant $C>0$ such that $A \leq C B$. $A \aeq B$ means $A \aleq B$ and $B \aleq A$. The constant $C$ can change from line to line and depends on dimension and exponent, but unless otherwise noted does not depend on $s$, $t$ etc.

\subsection{Mixed Measure Spaces}
Let $p,q \in (1,\infty)$, and consider the space $L^p(\ell^q)$ given by sequence $(f_j)_{j \in \Z} \subset L^p(\R^N)$ with finite norm
\[
 \|f_j\|_{L^p(\ell^q)} := \left \|\brac{\sum_j |f_j(x)|^q}^{\frac{1}{q}}\right \|_{L^p(\R^N,dx)}.
\]
By a slight abuse of notation, we will have the same notation when considering finite sequences $(f_j)_{j=-K}^K \subset L^p(\R^N)$.

From \cite[Theorem 1]{Benedek61} we obtain that $L^p(\ell^q)$ is a Banach space, and more importantly its dual space is $L^{p'}(\ell^{q'})$ in the following way: any linear functional $J \in (L^p(\ell^q))^\ast$ is given by an element $g \in L^{p'}(\ell^{q'})$ such that 
\[
 J(f) = \int_{\R^N} \sum_{j \in \Z} f_j(x) g_j(x)\, dx.
\]
In particular, from Hahn-Banach theorem, we have 
\begin{proposition}\label{pr:dualchar:Lplq} Let $p,q \in (1,\infty)$.
Let $(f_j)_{j \in \Z} \in L^{p} (\ell^{q})$. Then there exists $(g_j)_{j \in \Z} \subset L^{p'}(\R^N)$ with 
\[
 \|(g_j)_{j \in \Z}\|_{L^{p'}(\ell^{q'})} \leq 1
\]
and 
\[
 \|f_j\|_{L^p(\ell^q)} = \int_{\R^N} \sum_{j\in \Z} f_j(x) g_j(x)\, dx.
\]
An analogous statement holds for sequences $(f_j)_{j =-K}^K$.
\end{proposition}


\subsection{Fractional Laplacian}
For $s > 0$ denote by $\laps{s}$ the operator with Fourier symbol $|\xi|^s$, that is
\[
 \mathcal{F}(\laps{s} f)(\xi) := |\xi|^s \mathcal{F} f(\xi).
\]
It is well-known that there is an integral formula for the fractional Laplacian when $s \in (0,2)$, cf. \cite{Hitchhiker}. We need the following estimate on the constant that appears there.

\begin{lemma}\label{la:fraclap}
Let $f \in \Sw(\R^N)$ and $s \in (0,1)$. Then
\[
 (-\lap)^s f(x) = c_{N,s} \int_{\R^N} \frac{2f(x)-f(x+z)-f(x-z)}{|z|^{N+2s}}\, dz,
\]
where 
\[
 c_{N,s} \aeq \min\{s, 1-s\}.
\]
\end{lemma}
\begin{proof}
For $s \in (0,1)$ we have (see e.g. \cite{Hitchhiker})
\[
 (-\lap)^s f(x) = c_{N,s} \int_{\R^N} \frac{2f(x)-f(x+z)-f(x-z)}{|z|^{N+2s}}\, dz, \quad
 c_{N,s} = \frac{1}{2} \frac{4^{s} \Gamma(\frac{N}{2}+s)}{\pi^{\frac{N}{2}} |\Gamma(-s)|}.
\]
Since
\[
 \Gamma(-s) = -\frac{\pi}{\sin(\pi s)} \frac{1}{\Gamma(1+s)},
\]
for $s \in (0,1)$ we get 
$c_{N,s}   \aeq |\sin(\pi s)| \aeq \min\{s, 1-s\}.$
\end{proof}

\subsection{Littlewood-Paley projections and Triebel-Lizorkin Spaces}\label{s:triebelprelim}
{Below we will need to understand the space $L^p(\R^N)$ and the inhomogeneous Sobolev space
\[
H^{1,p}(\R^N) := \{f \in L^p(\R^N) \colon \nabla f \in L^p(\R^N)\}
\]
for $1 < p < \infty$, via Bessel potentials and Triebel-Lizorkin spaces.}

{First, recall the Bessel potential $(I-\Delta)^{s/2}$, given by the Fourier multiplier $(1+|\xi|^2)^{s/2}$ for $s \in \R$. We have $(I-\Delta)^{s/2} : \Sw(\R^N) \to \Sw(\R^N)$ continuously, thus $(I-\Delta)^{s/2}$ extends by duality to a map that acts on tempered distributions $\Sw'(\R^N)$.
For $1 < p < \infty$, it is known that $f \in \Sw'(\R^N)$ with $(I-\Delta)^{1/2} f \in L^p(\R^N)$ if and only if $f \in L^p(\R^N)$ with distributional gradient $\nabla f \in L^p(\R^N)$. This motivates one to define, for $s \in \R$ and $1 < p < \infty$, the space $H^{s,p}(\R^N)$, as the space of all tempered distributions $f \in \Sw'(\R^N)$ for which 
\[
\|f\|_{H^{s,p}(\R^N)} := \|(I-\Delta)^{s/2} f\|_{L^p(\R^N)} < \infty.
\]
When $s = 1$, the definition of $H^{s,p}(\R^N)$ agrees with our earlier definition in the previous paragraph using distributional gradients. We also have
\[
\|f\|_{H^{1,p}(\R^N)} \aeq_{p,N} \|f\|_{L^p(\R^N)} + \|\nabla f\|_{L^p(\R^N)}
\]
for $f \in H^{1,p}(\R^N)$.
}

{Next, for a function $f \in L^p(\R^N)$,} the $j$-th Littlewood-Paley projection is defined as
\[
 \Delta_j f(x) := {f*[2^{jN} \eta (2^j \cdot)](x).} 
\]
Here $\eta \in \Sw(\R^N)$ is a Schwartz function such that its Fourier transform $\mathcal{F} \eta$ satisfies
\begin{equation}\label{eq:sumpeq1}
 \sum_{j \in \Z} (\mathcal{F} \eta)(2^{j} \xi) = 1 \quad \forall \xi \neq 0.
\end{equation}
It is customary to assume $\eta \in \Sw(\R^N)$ is real-valued and symmetric in the sense that
\begin{equation}\label{eq:pjsymmetry}
\eta(-z) = \eta(z),
\end{equation}
so that $\Delta_j$ is a self-adjoint operator with respect to the $L^2(\R^N)$-scalar product.
We can and will also assume that 
{$\mathcal{F}\eta(\xi) = \mathcal{F}\eta_0(\xi)-\mathcal{F}\eta_0(2\xi)$ for some Schwartz function $\eta_0$ with $\mathcal{F}\eta_0(\xi) = 1$ for $|\xi| \leq 1$ and $\mathcal{F}\eta_0(\xi) = 0$ for $|\xi| \geq 2$.}
In particular 
\begin{equation}\label{eq:LPprojmv}
 \int_{\R^N} \eta(x)\ dx = c\, \mathcal{F}(\eta)(0) = 0.
\end{equation}
Also we have 
\begin{equation}\label{eq:Deltajsupp}
 \supp \mathcal{F}(\Delta_j f) \subset \left \{\xi \in \R^N: \quad 
 {\frac{1}{2}} \leq |2^{-j} \xi| \leq 2 \right \}.
\end{equation}
In particular, $\mathcal{F} \eta (2^j \xi) \mathcal{F} \eta (2^{j+\ell} \xi) = 0$ whenever $|\ell|\geq 2$, and thus
\begin{equation}\label{eq:LPproj}
 \Delta_j f(x) = \sum_{\ell = j-1}^{j+1} \Delta_j \Delta_\ell f(x).
\end{equation}
Then we have for any $f \in \Sw(\R^N)$, see \cite[Exercise 1.1.4]{Gmodern14},
\begin{equation}\label{eq:fpointwiselittlewoodpaley}
 f(x) = \sum_{j \in \Z} \Delta_j f(x) \quad \forall x \in \R^N,
\end{equation}
and the convergence is in 
{$L^p(\R^N)$} for any 
$p \in (1,\infty]$. 
{In particular, the set of all Schwartz functions whose Fourier transform is supported in a compact subset of $\R^N \setminus \{0\}$ is dense in $L^p(\R^N)$ if $1 < p < \infty$.}
For further reading on Littlewood-Paley projection we refer to \cite[6.2.2]{Gmodern14}. {Below we write $\Delta_{\leq 0}f$ for $f*\eta_0$.}

\begin{definition}[Triebel-Lizorkin {space}] 
Let $s \in \R$, $p,q \in (1,\infty)$. Then the {inhomogeneous}
Triebel-Lizorkin {space $F^s_{p,q}$} is defined as the set of all {tempered distributions $f \in \Sw'(\R^N)$ such that }
\[
\|f\|_{F^{s}_{p,q}(\R^N)} :=  \brac{\int_{\R^N} \brac{|\Delta_{\leq 0} f(x)|^q + \sum_{j \geq 1} 2^{jsq} |\Delta_j f(x)|^q}^{\frac{p}{q}}\, dx}^{\frac{1}{p}} < \infty.
\]
\end{definition}

{For a tempered distribution $f \in \Sw'(\R^N)$ we also define its homogeneous Triebel-Lizorkin semi-norm as follows:
\[
 [f]_{\dot{F}^{s}_{{p,}q}(\R^N)} := \brac{\int_{\R^N} \brac{\sum_{j \in \Z} 2^{jsq} |\Delta_j f(x)|^q}^{\frac{p}{q}}\, {dx}}^{\frac{1}{p}}.
\]
It is known, for instance, that if $s > 0$, $p,q \in (1,\infty)$ and $f \in F^s_{p,q}(\R^N)$, the homogeneous Triebel-Lizorkin semi-norm $[f]_{\dot{F}^s_{p,q}(\R^N)}$ is finite and $[f]_{\dot{F}^s_{p,q}(\R^N)} \lesssim_{s,p,q,N} \|f\|_{F^s_{p,q}(\R^N)}$.}

The class of Triebel-Lizorkin spaces and Besov spaces (where the role of integral and sum are reversed) contains several classical function spaces, we refer e.g. to \cite[\textsection 2.1.2, p.14]{RS96} or \cite[\textsection 2.3.5]{T83}. A well-known function space that is of Triebel-Lizorkin type is $L^p(\R^N)$ for $1 < p < \infty$: the theory of H\"{o}rmander-Mikhlin multipliers implies{, for $1 < p < \infty$, that $L^p(\R^N) = F^0_{p,2}(\R^N)$ and $H^{1,p}(\R^N) = F^1_{p,2}(\R^N)$ with equivalence of norms. Furthermore:}
\begin{lemma}[Littlewood-Paley]\label{la:LittlewoodPaley}
Let $p \in (1,\infty)$. Then 
for every $f \in L^p(\R^N)$ it holds
\[
\|f\|_{L^p(\R^N)} \aeq [f]_{{\dot{F}^{0}_{p,2}}(\R^N)}.
\]
Similarly, 
{a function $f \in L^p(\R^N)$ is in $H^{1,p}(\R^N)$, if and only if $[f]_{\dot{F}^1_{p,2}(\R^N)} < \infty$, in which case \[\|\nabla f\|_{L^p(\R^N)} \aeq [f]_{\dot{F}^1_{p,2}(\R^N)}.\]
The implicit constants in these equivalences depend only on $p$ and $N$.}
\end{lemma}

{
For $0 < s < 1$ and $1 < p < \infty$, we also have $H^{s,p}(\R^N) = F^s_{p,2}(\R^N)$, with 
\[
\|f\|_{H^{s,p}(\R^N)} \aeq \|f\|_{F^s_{p,2}(\R^N)}
\]
where the constants depend on $p$ and $N$ (and uniform \chgdarm{over} $s \in (0, 1)$).
}

{Recall the Gagliardo semi-norm
\[
[f]_{\dot{W}^{s,p}(\R^N)} = \left( \int_{\R^N} \int_{\R^N} \frac{|f(x)-f(y)|^p}{|x-y|^{N+sp}} dx dy \right)^{1/p}
\]
from the Introduction.}
It is also known that the following identification holds for $s \in (0,1)$, $p \in (1,\infty)$:
\[
 [f]_{\dot{F}^{s}_{p,p}(\R^N)} \aeq_{s,p,{N}} [f]_{{\dot{W}}^{s,p}(\R^N)} \quad \forall s \in (0,1), \ f \in \Sw(\R^N),
\]
and it is the objective of the present work to understand the dependency of the constant on $s$. It is important to observe that $F^{1}_{p,p}$ does \emph{not} correspond to the classical Sobolev space $H^{1,p}(\R^N) = F^{1}_{p,2}(\R^N)$, unless $p=2$.


We will need the following well-known vector-valued estimate for Littlewood-Paley projections, which follows from a vector-valued singular integral estimate (see e.g. \cite[Chapter II.5.4]{MR1232192}):

\begin{lemma}\label{lem:vv}
For any $1 < p < \infty$ and any $(f_j)_{j \in \Z} \in L^p(\ell^2)$, we have
\[
\brac{\int_{\R^N} \brac{\sum_{j \in \Z} |\Delta_j f_j(x)|^2}^{\frac{p}{2}}\, dx}^{\frac{1}{p}} 
\aleq  \brac{\int_{\R^N} \brac{\sum_{j \in \Z} |f_j(x)|^2}^{\frac{p}{2}}\, dx}^{\frac{1}{p}}
\]
\end{lemma}

{For $s > 0$ and $1 < p < \infty$, the Fourier multiplier $|\xi|^s (1+|\xi|^2)^{-s/2}$ defines a bounded linear map on $L^p(\R^N)$ (see \cite[Chapter V]{MR0290095}). Thus one can define a bounded linear map $(-\Delta)^{s/2} \colon H^{s,p}(\R^N) \to L^p(\R^N)$.
}
\chgdt{It is known that $\laps{s}: \dot{F}^{t+s}_{p,q} \to \dot{F}^{t}_{p,q}$ is an isomorphism, see \cite[2.6.2, Proposition 2]{RS96} and \cite[5.2.3]{T83}, \cite[2.3.8]{T83}. Their argument {(basically a vector-valued multiplier theorem)} implies:
\begin{lemma}\label{la:FspeqF0plaps}
Let $p,q \in (1,\infty)$, $\Theta > 0$. Then for any $s \in [0,\Theta]$ and any $f \in \Sw(\R^N)$ we have
\[
[f]_{\dot{F}^{s}_{p,q}} \aeq [\laps{s} f]_{\dot{F}^{0}_{p,q}}.
\]
Also
\[
[f]_{\dot{F}^{s}_{p,q}} \aleq [f]_{\dot{F}^{0}_{p,q}} + \left \|\brac{\sum_{j \geq 0} |\Delta_j \laps{s} f|^q}^{\frac{1}{q}}\right \|_{L^p(\R^N)}.
\]
The constant depends on $p,q,N$ and $\Theta$, and is otherwise independent of $s$.
\end{lemma}}

Next we need the following result about the Triebel-Lizorkin norm of a weak limit in $L^p$:

\begin{lemma}\label{la:triebelconv}
Let $f_k \in L^p(\R^N)$ weakly converge to $f \in L^p(\R^N)$, and assume that for some $s_k \uparrow t \in (0,\infty)$ we have 
\[
 \sup_{k} \, [f_k]_{\dot{F}^{s_k}_{{p,}q}{(\R^N)}} < \infty.
\]
Then 
\[
 [f]_{\dot{F}^{t}_{{p,}q}{(\R^N)}} \leq \sup_{k} \, [f_k]_{\dot{F}^{s_k}_{{p,}q}{(\R^N)}}.
\]
\end{lemma}
\begin{proof}

For each fixed $M$ and $R$, 
\[
\left\| \|2^{jt} \Delta_j f\|_{\ell^q(-M \leq j \leq M)} \right\|_{L^p(B(0,R))} = \lim_{k \to \infty} \left\| \|2^{js_k} \Delta_j f_k\|_{\ell^q(-M \leq j \leq M)} \right\|_{L^p(B(0,R))} \leq \sup_{k} [f_k]_{\dot{F}^{s_k}_{{p,}q}(\R^N)}.
\]
In the estimate above, the middle equality follows from the fact that
\[
\begin{split}
\brac{\sum_{j = -M}^M \abs{2^{j s_k} \Delta_j f_k-2^{j t} \Delta_j f}^q}^{\frac{1}{q}}
 \aleq&\max_{j \in \{-M,\ldots,M\}} |2^{j s_k}-2^{jt}|\, \max_{j \in \{-M,\ldots,M\}} |\Delta_j f|\, \\
 &+\max_{j \in \{-M,\ldots,M\}} 2^{j s_k} \, \max_{j \in \{-M,\ldots,M\}} |\Delta_j f_k-\Delta_j f|
 \end{split}
\]
which as $k \to \infty$ converges to 0 pointwise almost everywhere in $B(0,R)$, and hence in $L^p(B(0,R))$, if $f_k$ converges weakly to $f$ on $L^p(\R^N)$.
\end{proof}

\subsection{A duality characterization for Triebel-Lizorkin spaces}
The following duality statement must be known to experts -- we did not find it in this precise form in the literature, and thus repeat the proof.

\begin{theorem}[Duality]\label{th:dual}
Let $s \geq 0$, $p,q \in (1,\infty)$.
For any $f \in \Sw(\R^N)$ there exist $g \in \mathcal{F}^{-1}(C_c^\infty(\R^N \backslash \{0\}))$  such that
\[
 [g]_{\dot{F}^{s}_{p',q'}(\R^N)}  \leq 1
\]
and
\begin{equation}\label{eq:dual:goal}
[f]_{\dot{F}^{s}_{p,q}(\R^N)} \aeq \abs{\int_{\R^N} \laps{s}f(x)\ \laps{s} g(x) dx}.
\end{equation}
The constants depend on $s,p,q,N$, however if for some $\theta > 0$ we have 
$s \in [0,\theta)$, $p,q \in (1+\frac{1}{\theta},\theta)$,
then the constant can be chosen only to depend on $\theta$ and $N$.
\end{theorem}
Observe that in \eqref{eq:dual:goal}, $\laps{s} g$ belongs to the Schwartz class, since $\mathcal{F} g \in C_c^\infty(\R^N \backslash \{0\})$, consequently $\mathcal{F} \laps{s} g \in C_c^\infty(\R^N \backslash \{0\})$, which implies $\laps{s} g \in \Sw(\R^N)$. Moreover, it is easy to check that $f \in \Sw(\R^N)$ implies that $\laps{s}f \in L^\infty(\R^N)$, so the integral on the {right} hand side of \eqref{eq:dual:goal} makes sense.

\begin{proof}[Proof of \Cref{th:dual}]
Once $g$ is found, the $\ageq$-direction follows from two applications of H\"older's inequality and definition of the associated spaces.

So we focus on the $\aleq$-direction. Let $f \in \Sw(\R^N)$. Then by \chgdt{\Cref{la:FspeqF0plaps}}  
\[
 [f]_{\dot{F}^{s}_{p,q}(\R^N)} \aeq [\laps{s} f]_{\dot{F}^{0}_{p,q}(\R^N)}.
\]
In particular $(\Delta_j \laps{s} f)_{j \in \Z} \in L^p(\ell^q)$. In the case that $[f]_{\dot{F}^{s}_{p,q}} = 0$ we have $f$ is zero since the only polynomial in $\Sw(\R^N)$ is zero, and thus \eqref{eq:dual:goal} is trivially true for any $g$.

Consequently, from now own we assume
$[\laps{s} f]_{\dot{F}^{0}_{{p,}q}(\R^N)} > 0.$
By monotone convergence theorem, there must be $K \in \N$, depending on $f$, such that 
\[
 [\laps{s} f]_{\dot{F}^{0}_{{p,}q}(\R^N)} \leq 2\ \brac{\int_{\R^N} \brac{\sum_{j=-K}^{K} |\Delta_j \laps{s} f(x)|^q}^{\frac{p}{q}}\, dx}^{\frac{1}{p}}.
\]
Applying \Cref{pr:dualchar:Lplq}, there exists $(\bar{h}_j)_{j=-K}^{K} \subset L^{p'}(\R^N)$ with
\[
 \brac{\int_{\R^N} \brac{\sum_{j=-K}^{K} |\bar{h}_j(x)|^{q'}}^{\frac{p'}{q'}} dx}^{\frac{1}{p'}} \leq 1,
\]
such that 
\[
 [\laps{s} f]_{\dot{F}^{0}_{{p,}q}(\R^N)} \leq 2 \abs{\int_{\R^N} \sum_{j=-K}^{K} \Delta_j \laps{s} f(x)\, \bar{h}_j(x)\, dx }.
\]
By density of $C^{\infty}_c(\R^N)$ in $L^{p'}(\R^N)$, there exists $h_{-K}, \dots, h_K \in C^{\infty}_c(\R^N)$ such that
\[
\left \|\brac{\sum_{j=-K}^K |h_j-\bar{h}_j|^{q'}}^{\frac{1}{q'}} \right \|_{L^{p'}(\R^N)} \leq \frac{1}{4}.
\]
Consequently,
\begin{equation} \label{eq:hbdd}
 \brac{\int_{\R^N} \brac{\sum_{j=-K}^{K} |h_j(x)|^{q'}}^{\frac{p'}{q'}} dx}^{\frac{1}{p'}} \leq \frac{5}{4}
\end{equation}
and
\[
 [\laps{s} f]_{\dot{F}^{0}_{{p,}q}(\R^N)} \leq 2 \abs{\int_{\R^N} \sum_{j=-K}^{K} \Delta_j \laps{s} f(x)\, h_j(x)\, dx } + \frac{1}{2} [\laps{s} f]_{\dot{F}^{0}_{{p,}q}(\R^N)},
\]
which implies 
\[
 [\laps{s} f]_{\dot{F}^{0}_{{p,}q}(\R^N)} \leq 4 \abs{\int_{\R^N} \sum_{j=-K}^{K} \Delta_j \laps{s} f(x)\, h_j(x)\, dx }.
\]
With an integration by parts (in this case this is just Fubini's theorem, using also symmetry \eqref{eq:pjsymmetry}),
\[
 \int_{\R^N} \Delta_j \laps{s}f(x) h_j (x)\, dx = \int_{\R^N} \laps{s} f(x)\, \Delta_j h_j (x)\, dx.
\]
Now we set 
\[
 h := \sum_{j={-K}}^K \Delta_j h_j (x), \quad \text{and} \quad
 g :=  (-\Delta)^{-\frac{s}{2}} h.
\]
Then clearly $g \in \mathcal{F}^{-1}[C^{\infty}_c(\R^N \setminus \{0\})] \subset \Sw(\R^N)$, and the above shows that
\[
 [\laps{s} f]_{\dot{F}^{0}_{{p,}q}} \leq 4 \abs{\int_{\R^N} \laps{s} f(x)\, h(x)\, dx } = 4 \abs{\int_{\R^N} \laps{s} f(x)\, \laps{s} g(x)\, dx }.
\]
Furthermore, 
\[
[g]_{\dot F^s_{p,q}} \aeq [h]_{\dot F^0_{p,q}} \aeq \max_{\ell = -1,0,1} \brac{\int_{\R^N} \brac{\sum_{j=-K}^K |\Delta_{j+\ell} \Delta_j h_j(x)|^{q'}}^{\frac{p'}{q'}}dx}^{\frac{1}{p'}}
\]
By \Cref{lem:vv} and \eqref{eq:hbdd}, we then have $[g]_{\dot F^s_{p,q}} \aleq 1$. This completes the proof of this theorem.
\end{proof}

We also obtain the inhomogeneous version of \Cref{th:dual}.

\begin{theorem}[Inhomogeneous Duality Estimate]\label{th:dualinhom}
Let $s \geq 0$, $p,q \in (1,\infty)$.
For any $f \in \Sw(\R^N)$ there exist $g \in \Sw(\R^N)$ with $\mathcal{F} g$ supported on $\{|\xi| \geq 1/4\}$  such that 
\[
 [g]_{\dot{F}^{s}_{p',q'}}  \leq 1
\]
and
\[
[f]_{\dot{F}^{s}_{p,q}(\R^N)} \aleq [f]_{\dot{F}^{0}_{p,q}(\R^N)} + \abs{\int_{\R^N} \laps{s}f(x)\ \laps{s} g(x) \ dx}.
\]
The constants depend on $s,p,q,N$, however if for some $\theta > 0$ we have 
$s \in [0,\theta)$, $p,q \in (1+\frac{1}{\theta},\theta)$,
then the constant can be chosen only to depend on $\theta$ and $N$.
\end{theorem}
\begin{proof}
We have
\[
[f]_{\dot{F}^{s}_{{p,}q}(\R^N)} \leq [f]_{\dot{F}^{0}_{{p,}q}(\R^N)} + [\tilde{f}]_{\dot{F}^{s}_{{p,}q}(\R^N)} \quad \text{where} \quad
\tilde{f} := \sum_{k \geq 0} \Delta_k f.
\]
Following the proof of \Cref{th:dual}, one can find $\tilde{g} \in \Sw(\R^N)$, with $\mathcal{F}\tilde{g}$ supported on $\{|\xi| \geq 1/4\}$, such that $[\tilde{g}]_{\dot{F}^{s}_{p',q'}}  \leq 1$ and 
\[
[\tilde{f}]_{\dot{F}^{s}_{{p,}q}(\R^N)} \aleq \abs{\int_{\R^N} \laps{s}\tilde{f}(x)\ \laps{s} \tilde{g}(x) \ dx} = \abs{\int_{\R^N} \laps{s} f(x)\ \laps{s} \sum_{k \ge 0} \Delta_k \tilde{g}(x) \ dx}.
\]
It remains to check that $g := \sum_{k \geq 0} \Delta_k \tilde{g}$ satisfies the conclusion of the theorem.
\end{proof}

\section{An easy proof for the estimates for \texorpdfstring{$\dot{W}^{s,p}$ when $p=2$}{the estimates for Wsp when p=2}} \label{sect:p=2}
As a curiosity we give now a simple proof of the {equivalence between} ${\dot{W}}^{s,2}$
and {$\dot{F}^s_{2,2}$ seminorms.} 
\begin{proposition}\label{pr:peq2easyproof}
Let $s \in (0,1)$ and $f \in \Sw(\R^N)$. Then it holds with constants independent of $s$ and $f$,
\[
 \min\{s,(1-s)\}^{\frac{1}{2}} [f]_{{\dot{W}}^{s,2}(\R^N)} \aeq [f]_{\dot{F}^{s}_{2,2}(\R^N)}.
\]
\end{proposition}
\begin{proof}
Let $f \in \Sw(\R^N)$. Then by \chgdt{\Cref{la:FspeqF0plaps}} 
and Fubini's theorem,
\[
 [f]_{\dot{F}^{s}_{2,2}(\R^N)}^2 \aeq [\laps{s}f]_{\dot{F}^{0,2}_2(\R^N)}^2 = \sum_{j \in \Z} \int_{\R^N} \laps{s} \Delta_j f(x) \, \overline{\laps{s} \Delta_j f(x)} \ dx.
\]
Integrating by parts (via the Fourier transform) we have 
\[
 \int_{\R^N} \laps{s} \Delta_j f(x)\, \overline{ \laps{s} \Delta_j f(x)} \ dx = \int_{\R^N} \laps{2s} \Delta_j f(x) \, \overline{\Delta_j f(x)} \ dx.
\]
With the integral characterization of the fractional Laplacian, \Cref{la:fraclap}, we have 
\begin{equation} \label{eq:L2temp}
 \int_{\R^N} \laps{2s} \Delta_j f(x) \, \overline{ \Delta_j f(x)} \ dx = c_{N,s} \int_{\R^N} \int_{\R^N} \frac{(2\Delta_j f(x)-\Delta_j f(x+z)-\Delta_j f(x-z))\,\overline{ \Delta_j f(x)} }{|z|^{N+2s}}\, dz\, dx.
\end{equation}
We note that by a change of variables $x \mapsto x+z$, we have 
\begin{equation} \label{eq:shift}
\int_{\R^N} ( \Delta_j f(x) - \Delta_j f(x-z) ) \overline{\Delta_j f(x)} dx = \int_{\R^N} (\Delta_j f(x+z) - \Delta_j f(x)) \overline{\Delta_j f(x+z)} dx.
\end{equation}
Hence \eqref{eq:L2temp} is equal to
\[
\begin{split}
c_{N,s} \int_{\R^N} \int_{\R^N} \frac{|\Delta_j f(x+z)-\Delta_jf(x)|^2}{|z|^{N+2s}}\, dx\, dz
= c_{N,s} \int_{\R^N} \frac{\|\Delta_j f(\cdot+z)-\Delta_j f(\cdot)\|_{L^2(\R^N)}^2}{ |z|^{N+2s}} dz.
\end{split}
\]
We obtain via \Cref{la:LittlewoodPaley}
\[
 \sum_{j \in \Z} \|\Delta_j f(\cdot+z)-\Delta_j f(\cdot)\|_{L^2(\R^N)}^2 = \|f(\cdot+z)-f(\cdot)\|_{L^2(\R^N)}^2,
\]
from which we deduce
\[
  [f]_{\dot{F}^{s}_{2,2}(\R^N)} \aeq  \brac{c_{N,s} \int_{\R^N} \frac{ \|f(\cdot+z)-f(\cdot)\|_{L^2(\R^N)}^2 }{|z|^{N+2s}} dz}^{\frac{1}{2}}  = c_{N,s}^{\frac{1}{2}} [f]_{{\dot{W}}^{s,2}(\R^N)}.
\]
The proposition then follows from the estimate $c_{N,s} \aeq \min\{s,(1-s)\}$ in Lemma~\ref{la:fraclap}.
\end{proof}

\section{The upper bounds {for \texorpdfstring{$[f]_{\dot{W}^{s,p}}$}{[f]{Wsp}} in} Theorems~\ref{th:mainFspp} and \ref{th:mainFsp2}} \label{sect:upper1.21.5}
In this section we prove the upper bounds {for $[f]_{\dot{W}^{s,p}}$ in} Theorems~\ref{th:mainFspp} and \ref{th:mainFsp2}, namely we show

\begin{theorem}
Let $p \in (1,\infty)$, $s \in (0,1)$ and $f \in \Sw(\R^N)$. Then
\begin{equation}\label{eq:upper1}
 [f]_{{\dot{W}}^{s,p}(\R^N)} \aleq  \brac{\frac{1}{s^{\frac{1}{p}}} + \frac{1}{(1-s)^{\frac{1}{p}}}}\, [f]_{\dot{F}^{s}_{p,p}(\R^N)} \quad \text{if $1 < p \leq 2$},
\end{equation}
\begin{equation}\label{eq:upper2}
 [f]_{{\dot{W}}^{s,p}(\R^N)} \aleq \brac{\frac{1}{s^{\frac{1}{2}}} + \frac{1}{(1-s)^{\frac{1}{2}}}}\, [f]_{\dot{F}^{s}_{p,p}(\R^N)} \quad \text{if $2 \leq p < \infty$},
\end{equation}
and
\begin{equation}\label{eq:upper3}
 [f]_{{\dot{W}}^{s,p}(\R^N)}  \aleq \brac{\frac{1}{s^{\frac{1}{p}}} + \frac{1}{(1-s)^{\frac{1}{p}}}}\, [f]_{\dot{F}^{s}_{p,2}(\R^N)} \quad \text{if $2 \leq p < \infty$}.
\end{equation}
\end{theorem}

\eqref{eq:upper1} and \eqref{eq:upper2} have been proven by Gu and the third author in \cite{GY21}, and \eqref{eq:upper3} is a slight adaptation of their argument. We still present it for the sake of completeness.

Below we repeatedly use the following estimate for geometric sums: for $1 < p < \infty$,
\begin{equation} \label{eq:geoms>0}
 \sum_{j \geq 0} 2^{-jsp} =  \frac{1}{1-2^{-sp}} \aeq \frac{1}{s} \quad \text{for $s > 0$}
\end{equation}
and
\begin{equation} \label{eq:geoms<1}
 \sum_{j \leq 0} 2^{j \sigma p} =  \frac{1}{1-2^{-\sigma p}} \aeq \frac{1}{\sigma} \quad \text{for $\sigma > 0$}.
\end{equation}

The first step for \eqref{eq:upper1}, \eqref{eq:upper2} and \eqref{eq:upper3} is the following estimate.
\begin{lemma}\label{la:firststep}
Let $p \in (1,\infty)$ and $s \in (0,1)$. Then
\[
\begin{split}
 [f]_{{\dot{W}}^{s,p}(\R^N)} \aleq& \brac{\sum_{k \in \Z} \int_{\R^N}\brac{\sum_{j \geq 0} |2^{ks} \Delta_{k+j} f(x)|^2}^{\frac{p}{2}} dx}^{\frac{1}{p}} +\brac{\sum_{k \in \Z} \int_{\R^N}\brac{\sum_{j \leq 0} \abs{2^{j} 2^{ks} \Delta_{k+j} f(x)}^2}^{\frac{p}{2}} dx}^{\frac{1}{p}}.
 \end{split}
\]
\end{lemma}
\begin{proof}
We have 
\[
\begin{split}
 [f]_{{\dot{W}}^{s,p}(\R^N)}
=& \brac{\int_{\R^N} \frac{ \|f(\cdot+z)-f(\cdot)\|_{L^p(\R^N)}^p }{|z|^{N+sp}} dz }^{\frac{1}{p}}
 \aleq \brac{\sum_{k \in \Z} 2^{ksp} \sup_{|z| \aeq 2^{-k}} \|f(\cdot+z)-f(\cdot)\|_{L^p(\R^N)}^p}^{\frac{1}{p}}.
  \end{split}
\]
But for $|z| \aeq 2^{-k}$, Littlewood-Paley implies
\[
\begin{split}
 \|f(\cdot+z)-f(\cdot)\|_{L^p(\R^N)}
 \aleq& \brac{ \int_{\R^N}\brac{\sum_{j \in \Z} |\Delta_{k+j} f(x+z)-\Delta_{k+j} f(x)|^2}^{\frac{p}{2}} dx }^{\frac{1}{p}} 
 \end{split}
\]
which by the triangle inequality is 
\[
\begin{split}
 \aleq& \brac{ \int_{\R^N}\brac{\sum_{j \ge 0} |\Delta_{k+j} f(x+z)-\Delta_{k+j} f(x)|^2}^{\frac{p}{2}} dx }^{\frac{1}{p}} + \brac{ \int_{\R^N}\brac{\sum_{j < 0} |\Delta_{k+j} f(x+z)-\Delta_{k+j} f(x)|^2}^{\frac{p}{2}} dx }^{\frac{1}{p}}.
\end{split}
\]
The first term above is bounded by the triangle inequality by
\[
\begin{split}
2 \brac{ \int_{\R^N}\brac{\sum_{j \ge 0} |\Delta_{k+j} f(x)|^2}^{\frac{p}{2}} dx }^{\frac{1}{p}}.
\end{split}
\]
For the second term, the fundamental theorem of calculus implies
\[
\begin{split}
|\Delta_{k+j} f(x+z) - \Delta_{k+j} f(x)| 
\aleq |z| \int_0^1 |\nabla \Delta_{k+j} f(x+tz)| dt,
\end{split}
\]
so Minkowski's inequality implies
\[
\begin{split}
\brac{ \int_{\R^N}\brac{\sum_{j < 0} |\Delta_{k+j} f(x+z)-\Delta_{k+j} f(x)|^2}^{\frac{p}{2}} dx }^{\frac{1}{p}}
\aleq & |z|\int_0^1 \brac{ \int_{\R^N}\brac{\sum_{j < 0} |\nabla \Delta_{k+j} f(x+tz)|^2}^{\frac{p}{2}} dx }^{\frac{1}{p}} dt
\end{split}  
\]
which is
\[
\begin{split}
\aeq & 2^{-k} \brac{ \int_{\R^N}\brac{\sum_{j < 0} |\nabla \Delta_{k+j} f(x)|^2}^{\frac{p}{2}} dx }^{\frac{1}{p}} \aleq \brac{ \int_{\R^N}\brac{\sum_{j < 0} |2^{j} \Delta_{k+j} f(x)|^2}^{\frac{p}{2}} dx }^{\frac{1}{p}}
\end{split}  
\]
by the Littlewood-Paley inequality again. Altogether, we get
\[
2^{ksp} \sup_{|z| \aeq 2^{-k}} \|f(\cdot+z)-f(\cdot)\|_{L^p(\R^N)}^p \aleq \int_{\R^N}\brac{\sum_{j \ge 0} |2^{ks} \Delta_{k+j} f(x)|^2}^{\frac{p}{2}} dx +  \int_{\R^N}\brac{\sum_{j < 0} |2^{j} 2^{ks} \Delta_{k+j} f(x)|^2}^{\frac{p}{2}} dx
\]
which implies the desired estimate.
\end{proof}

Now \eqref{eq:upper1} is a consequence of \Cref{la:firststep} and the following proposition.
\begin{proposition}
Let $1 < p \leq 2$. Then
\[
 \brac{\sum_{k \in \Z} \int_{\R^N}\brac{\sum_{j \geq 0} |2^{ks} \Delta_{k+j} f(x)|^2}^{\frac{p}{2}} dx}^{\frac{1}{p}} \aleq \frac{1}{s^{\frac{1}{p}}} [f]_{\dot{F}^{s}_{p,p}(\R^N)}.
 \]
and 
\[
  \brac{\sum_{k \in \Z} \int_{\R^N}\brac{\sum_{j < 0} \abs{2^j 2^{ks} \Delta_{k+j} f(x)}^2}^{\frac{p}{2}} dx}^{\frac{1}{p}} \aleq \frac{1}{(1-s)^{\frac{1}{p}}} [f]_{\dot{F}^{s}_{p,p}(\R^N)}.
\]
\end{proposition}
\begin{proof}
Since $p \in (1,2]$, we have $\abs{\sum_j F_j}^{\frac{p}{2}} \leq \sum_j |F_j|^{\frac{p}{2}}$. Thus
\[
 \brac{\sum_{k \in \Z} \int_{\R^N}\brac{\sum_{j \geq 0} |2^{ks} \Delta_{k+j} f(x)|^2}^{\frac{p}{2}} dx}^{\frac{1}{p}} \leq  \brac{\sum_{k \in \Z} \int_{\R^N} \sum_{j \geq 0} |2^{ks} \Delta_{k+j} f(x)|^p dx}^{\frac{1}{p}}, 
\] 
and we conclude by noting that for $s > 0$, \eqref{eq:geoms>0} gives
\[
\sum_{k \in \Z} \sum_{j \geq 0} |2^{ks} \Delta_{k+j} f(x)|^p = \sum_{j \geq 0} 2^{-jsp} \sum_{k \in \Z} |2^{(k+j)s} \Delta_{k+j} f(x)|^p \aeq \frac{1}{s} \sum_{k \in \Z} | 2^{ks} \Delta_k f(x)|^p;
\]
similarly 
\[
\brac{\sum_{k \in \Z} \int_{\R^N}\brac{\sum_{j < 0} \abs{2^j 2^{ks} \Delta_{k+j} f(x)}^2}^{\frac{p}{2}} dx}^{\frac{1}{p}} \leq \brac{\sum_{k \in \Z} \int_{\R^N} \sum_{j < 0} |2^j 2^{ks} \Delta_{k+j} f(x)|^p dx}^{\frac{1}{p}},
\]
and we conclude by noting that for $s < 1$, \eqref{eq:geoms<1} with $\sigma = 1-s$ gives
\[
\sum_{k \in \Z} \sum_{j < 0} |2^j 2^{ks} \Delta_{k+j} f(x)|^p = \sum_{j < 0} 2^{j(1-s)p} \sum_{k \in \Z} |2^{(k+j)s} \Delta_{k+j} f(x)|^p \aeq \frac{1}{1-s} \sum_{k \in \Z} | 2^{ks} \Delta_k f(x)|^p.
\]
\end{proof}

Next, \eqref{eq:upper2} is a consequence \Cref{la:firststep} and the following
\begin{proposition}
Let $2 \leq p < \infty$. Then
\[
 \brac{\sum_{k \in \Z} \int_{\R^N}\brac{\sum_{j \geq 0} |2^{ks} \Delta_{k+j} f(x)|^2}^{\frac{p}{2}} dx}^{\frac{1}{p}} \aleq \frac{1}{s^{\frac{1}{2}}}\, [f]_{\dot{F}^{s}_{p,p}(\R^N)} 
 \]
 and 
 \[
 \brac{\sum_{k \in \Z} \int_{\R^N}\brac{\sum_{j < 0} \abs{2^{j} 2^{ks} \Delta_{k+j} f(x)}^2}^{\frac{p}{2}} dx}^{\frac{1}{p}} \aleq \frac{1}{(1-s)^{\frac{1}{2}}}\, [f]_{\dot{F}^{s}_{p,p}(\R^N)}.
 \]
\end{proposition}
\begin{proof}
Since $p \geq 2$ we can apply Minkowski inequality for $\ell^{\frac{p}{2}} {(} L^{\frac{p}{2}}(\R^N) {)}$ and get
\[
 \brac{\sum_{k \in \Z} \int_{\R^N}\brac{\sum_{j \geq 0} |2^{ks} \Delta_{k+j} f(x)|^2}^{\frac{p}{2}} dx}^{\frac{1}{p}} \leq 
 \brac{ \sum_{j \geq 0} \brac{\sum_{k \in \Z} \int_{\R^N} |2^{ks} \Delta_{k+j} f(x)|^p dx}^{\frac{2}{p}} }^{\frac{1}{2}}
\]
which for $s > 0$ is equal to 
\[
 \brac{ \sum_{j \geq 0} 2^{-2js} \brac{\sum_{k \in \Z} \int_{\R^N} |2^{(k+j)s} \Delta_{k+j} f(x)|^p dx}^{\frac{2}{p}} }^{\frac{1}{2}} = \brac{ \sum_{j \geq 0} 2^{-2js} [f]_{\dot{F}^s_{p,p}}^2 }^{\frac{1}{2}} \aeq \frac{1}{s^{\frac{1}{2}}} [f]_{\dot{F}^s_{p,p}}
\]
using \eqref{eq:geoms>0} with $p = 2$. Similarly,
\[
 \brac{\sum_{k \in \Z} \int_{\R^N}\brac{\sum_{j < 0} \abs{2^{j} 2^{ks} \Delta_{k+j} f(x)}^2}^{\frac{p}{2}} dx}^{\frac{1}{p}} \leq 
 \brac{ \sum_{j < 0} \brac{ \sum_{k \in \Z} \int_{\R^N} \abs{2^{j} 2^{ks} \Delta_{k+j} f(x)}^p dx}^{\frac{2}{p}} }^{\frac{1}{2}}
\]
which for $s < 1$ is equal to
\[
 \brac{ \sum_{j < 0} 2^{2j(1-s)} \brac{\sum_{k \in \Z} \int_{\R^N} |2^{(k+j)s} \Delta_{k+j} f(x)|^p dx}^{\frac{2}{p}} }^{\frac{1}{2}} = \brac{ \sum_{j < 0} 2^{2j(1-s)} [f]_{\dot{F}^s_{p,p}}^2 }^{\frac{1}{2}} \aeq \frac{1}{(1-s)^{\frac{1}{2}}} [f]_{\dot{F}^s_{p,p}}
\]
using \eqref{eq:geoms<1} with $\sigma = 1-s$ and $p = 2$. 
\end{proof}

Lastly, \eqref{eq:upper3} is a consequence of \Cref{la:firststep} and the following proposition.
\begin{proposition}
Let $2 \leq p < \infty$. Then
\[
 \brac{\sum_{k \in \Z} \int_{\R^N}\brac{\sum_{j \geq 0} |2^{ks} \Delta_{k+j} f(x)|^2}^{\frac{p}{2}} dx}^{\frac{1}{p}}  \aleq \frac{1}{s^{\frac{1}{p}}}\, [f]_{\dot{F}^{s}_{p,2}}
 \]
 and 
 \[
 \brac{\sum_{k \in \Z} \int_{\R^N}\brac{\sum_{j < 0} \abs{2^{j} 2^{ks} \Delta_{k+j} f(x)}^2}^{\frac{p}{2}} dx}^{\frac{1}{p}} \aleq \frac{1}{(1-s)^{\frac{1}{p}}}\, [f]_{\dot{F}^{s}_{p,2}}.
 \]
\end{proposition}
\begin{proof}
Fix $x \in \R^N$. We have for any $k \in \Z$
\[
 \sum_{j \geq 0} |2^{ks} \Delta_{k+j} f(x)|^2 \leq \sum_{j \geq 0} |2^{(k+j)s} \Delta_{k+j} f(x)|^2 \leq \sum_{\ell \in \Z} |2^{\ell s} \Delta_\ell f(x)|^2.
\]
Consequently, since $p \geq 2$ we have 
\[
\sum_{k \in \Z} \brac{\sum_{j \geq 0} |2^{ks} \Delta_{k+j} f(x)|^2}^{\frac{p}{2}}
\leq \sum_{k \in \Z} \sum_{j \geq 0} |2^{ks} \Delta_{k+j} f(x)|^2 \brac{\sum_{\ell \in \Z} |2^{\ell s} \Delta_\ell f(x)|^2}^{\frac{p}{2}-1}
\]
which is
\[
=\sum_{j \geq 0} 2^{-2js} \sum_{k \in \Z} |2^{(k+j)s} \Delta_{k+j} f(x)|^2 \brac{\sum_{\ell \in \Z} |2^{\ell s} \Delta_\ell f(x)|^2}^{\frac{p}{2}-1} 
\aeq \frac{1}{s} \brac{\sum_{\ell \in \Z} |2^{\ell s} \Delta_\ell f(x)|^2}^{\frac{p}{2}} 
\]
using \eqref{eq:geoms>0} with $p = 2$. Integrating this with respect to $x$ gives the first inequality. Similarly, for any $k \in \Z$
\[
 \sum_{j < 0} |2^{j} 2^{ks} \Delta_{k+j} f(x)|^2 \leq \sum_{j < 0} |2^{(k+j)s} \Delta_{k+j} f(x)|^2 \leq \sum_{\ell \in \Z} |2^{\ell s} \Delta_\ell f(x)|^2.
\]
Consequently, since $p \geq 2$ we have 
\[
\sum_{k \in \Z} \brac{\sum_{j < 0} |2^{ks} \Delta_{k+j} f(x)|^2}^{\frac{p}{2}}
\leq \sum_{k \in \Z} \sum_{j < 0} |2^{ks} \Delta_{k+j} f(x)|^2 \brac{\sum_{\ell \in \Z} |2^{\ell s} \Delta_\ell f(x)|^2}^{\frac{p}{2}-1}
\]
which is
\[
=\sum_{j < 0} 2^{2j(1-s)} \sum_{k \in \Z} |2^{(k+j)s} \Delta_{k+j} f(x)|^2 \brac{\sum_{\ell \in \Z} |2^{\ell s} \Delta_\ell f(x)|^2}^{\frac{p}{2}-1} 
\aeq \frac{1}{1-s} \brac{\sum_{\ell \in \Z} |2^{\ell s} \Delta_\ell f(x)|^2}^{\frac{p}{2}} 
\]
using \eqref{eq:geoms<1} with $\sigma = 1-s$ and $p = 2$. Integrating this with respect to $x$ gives the second inequality. 
\end{proof}

\section{The upper bound {for \texorpdfstring{$[f]_{\dot{W}^{s,p}}$}{[f]{Wsp}} in} of Theorem~\ref{th:sobolev}}\label{s:sobolevbound}
In this section we prove the first part of \Cref{th:sobolev}{, which provides an upper bound for $[f]_{\dot{W}^{s,p}}$ in terms of $[f]_{\dot{F}^r_{p,2}}$ and $[f]_{\dot{F}^t_{p,2}}$ when $0 \leq r < s < t \leq 1$ and $1 < p < \infty$}. Namely, we show that for any {such $r,s,t,p$} and $f \in \Sw(\R^N)$,
\[
\tag{\ref{eq:sobolev1}} [f]_{{\dot{W}}^{s,p}(\R^N)}  \aleq \frac{1}{(s-r)^{\frac{1}{p}}}\, [f]_{\dot{F}^{r}_{p,2}(\R^N)} + \frac{1}{(t-s)^{\frac{1}{p}}}\, [f]_{\dot{F}^{t}_{p,2}(\R^N)}.
\]
The constant $C$ depends on $p$ and $N$ only. 
In view of \Cref{la:firststep},
\eqref{eq:sobolev1} is then a consequence of the following four lemmata.

\begin{lemma}
Let $p \in (1,\infty)$, $0 \le r < s$. Then 
 \[
  \brac{\sum_{k < 0} \int_{\R^N}\brac{\sum_{j \geq 0} |2^{ks} \Delta_{k+j} f(x)|^2}^{\frac{p}{2}} dx}^{\frac{1}{p}}
  \aleq \frac{1}{(s-r)^{\frac{1}{p}}} [f]_{\dot{F}^{r}_{p,2}(\R^N)}. 
 \]
\end{lemma}
\begin{proof}
We have for any $0 \le r < s$
\[
\begin{split}
\sum_{k < 0} \brac{\sum_{j \geq 0} \abs{ 2^{ks} \Delta_{k+j} f(x)}^2}^{\frac{p}{2}} &\leq \sum_{k < 0} \brac{\sum_{j \geq 0} \abs{2^{jr} 2^{ks} \Delta_{k+j} f(x)}^2}^{\frac{p}{2}} \\
&= \sum_{k < 0} 2^{k(s-r)p} \brac{\sum_{j < 0} \abs{2^{(k+j)r} \Delta_{k+j} f(x)}^2}^{\frac{p}{2}}.
\end{split}
\]
We extend the sum over $j$ to all integers, and use \eqref{eq:geoms<1} with $s-r$ in place of $\sigma$ to evaluate the sum over $k$. This gives
\[
\begin{split}
\sum_{k < 0} \brac{\sum_{j \geq 0} \abs{ 2^{ks} \Delta_{k+j} f(x)}^2}^{\frac{p}{2}} 
&\aleq \frac{1}{s-r} \brac{\sum_{\ell \in \Z} \abs{2^{\ell r} \Delta_{\ell} f(x)}^2}^{\frac{p}{2}},
\end{split}
\]
which gives the conclusion of the lemma upon integrating in $x$.
\end{proof}

\begin{lemma}
Let $p \in (1,\infty)$ and $s < t \leq 1$. Then 
\[
 \brac{\sum_{k \geq 0} \int_{\R^N} \brac{\sum_{j < 0} \abs{2^{j} 2^{ks} \Delta_{k+j} f(x)}^2}^{\frac{p}{2}} dx}^{\frac{1}{p}} \aleq \frac{1}{(t-s)^{\frac{1}{p}}}\, [f]_{\dot{F}^{t}_{p,2}(\R^N)}.
\]
\end{lemma}
\begin{proof}
We have for any $s < t \leq 1$
\[
\begin{split}
\sum_{k \geq 0} \brac{\sum_{j < 0} \abs{2^{j} 2^{ks} \Delta_{k+j} f(x)}^2}^{\frac{p}{2}} &\leq \sum_{k \geq 0} \brac{\sum_{j < 0} \abs{2^{jt} 2^{ks} \Delta_{k+j} f(x)}^2}^{\frac{p}{2}} \\
&= \sum_{k \geq 0} 2^{-k(t-s)p} \brac{\sum_{j < 0} \abs{2^{(k+j)t} \Delta_{k+j} f(x)}^2}^{\frac{p}{2}}.
\end{split}
\]
We extend the sum over $j$ to all integers, and use \eqref{eq:geoms>0} with $t-s$ in place of $s$ to evaluate the sum over $k$. This gives
\[
\begin{split}
\sum_{k \geq 0} \brac{\sum_{j < 0} \abs{2^{j} 2^{ks} \Delta_{k+j} f(x)}^2}^{\frac{p}{2}} 
&\aleq \frac{1}{t-s} \brac{\sum_{\ell \in \Z} \abs{2^{\ell t} \Delta_{\ell} f(x)}^2}^{\frac{p}{2}},
\end{split}
\]
which gives the conclusion of the lemma upon integrating in $x$.
\end{proof}

\begin{lemma}
Let $p \in (1,\infty)$, $r < s$ and $r \le 1$. Then 
\[
 \brac{\sum_{k \leq 0} \int_{\R^N}\brac{\sum_{j < 0} \abs{2^j 2^{ks} \Delta_{k+j} f(x)}^2}^{\frac{p}{2}} dx}^{\frac{1}{p}} \aleq \frac{1}{(s-r)^{\frac{1}{p}}}\, [f]_{\dot{F}^{r}_{p,2}(\R^N)}.
\]
\end{lemma}
\begin{proof}
We have for any $r \le 1$ and $s > r$
\[
\begin{split}
\sum_{k \leq 0} \brac{\sum_{j < 0} \abs{2^{j} 2^{ks} \Delta_{k+j} f(x)}^2}^{\frac{p}{2}} &\leq \sum_{k \leq 0} \brac{\sum_{j < 0} \abs{2^{jr} 2^{ks} \Delta_{k+j} f(x)}^2}^{\frac{p}{2}} \\
&= \sum_{k \leq 0} 2^{k(s-r)p} \brac{\sum_{j < 0} \abs{2^{(k+j)r} \Delta_{k+j} f(x)}^2}^{\frac{p}{2}}.
\end{split}
\]
We extend the sum over $j$ to all integers, and use \eqref{eq:geoms<1} with $s-r$ in place of $\sigma$ to evaluate the sum over $k$. This gives
\[
\begin{split}
\sum_{k \leq 0} \brac{\sum_{j < 0} \abs{2^{j} 2^{ks} \Delta_{k+j} f(x)}^2}^{\frac{p}{2}} 
&\aleq \frac{1}{s-r} \brac{\sum_{\ell \in \Z} \abs{2^{\ell r} \Delta_{\ell} f(x)}^2}^{\frac{p}{2}},
\end{split}
\]
which gives the conclusion of the lemma upon integrating in $x$.
\end{proof}

\begin{lemma}
Let $p \in (1,\infty)$, $s < t$ and $t \ge 0$. Then 
 \[
  \brac{\sum_{k \geq 0} \int_{\R^N}\brac{\sum_{j \geq 0} |2^{ks} \Delta_{k+j} f(x)|^2}^{\frac{p}{2}} dx}^{\frac{1}{p}}
  \aleq \frac{1}{(t-s)^{\frac{1}{p}}} [f]_{\dot{F}^{t,p}_{2}(\R^N)}. 
 \]
\end{lemma}
\begin{proof}
We have for any $t \ge 0$ and $s < t$
\[
\begin{split}
\sum_{k \geq 0} \brac{\sum_{j \geq 0} \abs{ 2^{ks} \Delta_{k+j} f(x)}^2}^{\frac{p}{2}} &\leq \sum_{k \geq 0} \brac{\sum_{j \geq 0} \abs{2^{jt} 2^{ks} \Delta_{k+j} f(x)}^2}^{\frac{p}{2}} \\
&= \sum_{k \geq 0} 2^{-k(t-s)p} \brac{\sum_{j < 0} \abs{2^{(k+j)t} \Delta_{k+j} f(x)}^2}^{\frac{p}{2}}.
\end{split}
\]
We extend the sum over $j$ to all integers, and use \eqref{eq:geoms>0} with $t-s$ in place of $s$ to evaluate the sum over $k$. This gives
\[
\begin{split}
\sum_{k \geq 0} \brac{\sum_{j \geq 0} \abs{ 2^{ks} \Delta_{k+j} f(x)}^2}^{\frac{p}{2}}  
&\aleq \frac{1}{t-s} \brac{\sum_{\ell \in \Z} \abs{2^{\ell t} \Delta_{\ell} f(x)}^2}^{\frac{p}{2}},
\end{split}
\]
which gives the conclusion of the lemma upon integrating in $x$.
\end{proof}

\section{The lower bounds {for \texorpdfstring{$[f]_{\dot{W}^{s,p}}$}{[f]{Wsp}}}: proof via duality} \label{sect:lower}

We obtain the lower bounds {for $[f]_{\dot{W}^{s,p}}$ in \Cref{th:mainFspp}, \Cref{th:mainFsp2} and \Cref{th:sobolev}} from the {corresponding} upper bounds by a duality argument, and using the integral representation of the fractional Laplacian, adapting the proof of  \Cref{pr:peq2easyproof}.

Our main ingredient is the following duality estimate.
\begin{proposition}\label{pr:dualityforlower}
Let $p,q \in (1,\infty)$, $s \in (0,1)$. Let $t_1,t_2 >0$, such that $t_1 + t_2 = 2s$. Let $p_1,p_2 \in (1,\infty)$, such that $\frac{1}{p_1} + \frac{1}{p_2} = 1$. Then for any $f \in \Sw(\R^N)$ we have 
\[
 [f]_{\dot{F}^{s}_{p,q}(\R^N)} \aleq  \min\{s,(1-s)\}\, [f]_{{\dot{W}}^{t_1,p_1}(\R^N)}\, \sup_{[g]_{\dot{F}^{s}_{p',q'}} \leq 1} [g]_{{\dot{W}}^{t_2,p_2}(\R^N)},
\]
where the supremum on the right-hand side is over Schwartz functions $g\in \mathcal{F}^{-1}(C_c^\infty(\R^N \backslash \{0\}))$. 

We also have
\[
 [f]_{\dot{F}^{s}_{p,q}(\R^N)} \aleq  [f]_{\dot{F}^{0}_{{p,q}}(\R^N)} + \min\{s,(1-s)\}\, [f]_{{\dot{W}}^{t_1,p_1}(\R^N)}\, \sup_{\substack{[g]_{\dot{F}^{s}_{p',q'}} \leq 1 \\ \text{supp} \mathcal{F}g \subset \{|\xi| \geq 1/4\}}} [g]_{{\dot{W}}^{t_2,p_2}(\R^N)}
\]
where this time the supremum on the right-hand side is over Schwartz functions $g$ with the indicated constraints. 
\end{proposition}
\begin{proof}
By duality, \Cref{th:dual}, for any $f \in \Sw(\R^N)$ there exists $g \in \mathcal{F}^{-1}(C_c^\infty(\R^N \backslash \{0\}))$ with $[g]_{\dot{F}^{s}_{p',q'}} \leq 1$ and 
\begin{equation}\label{eq:dualfl:34}
 [f]_{\dot{F}^{s}_{{p,q}}(\R^N)} \aleq \abs{\int_{\R^N} \laps{s} f(x) \laps{s} g(x) \ dx}.
\end{equation}
Integrating by parts the operator $\laps{s}$ (this can be done via Fourier transform and Plancherel, since $f,g \in \Sw(\R^N)$) we find for a constant $C_{\mathcal{F}}$ depending on the precise choice of Fourier transform
\[
\begin{split}
\int_{\R^N} \laps{s} f \laps{s} g =& C_{\mathcal{F}}\int_{\R^N} |\xi|^{s} \mathcal{F}f(\xi)\, |\xi|^s \mathcal{F}g(-\xi)\, {d\xi}\\
=& C_{\mathcal{F}}\int_{\R^N} |\xi|^{2s} \mathcal{F}f(\xi)\, \mathcal{F}g(-\xi)\, {d\xi}=\int_{\R^N} (-\lap)^{s} f\, g.
\end{split}
\]
From \Cref{la:fraclap} and symmetry arguments we then find with a constant $c_{N,s}$ such that $c_{N,s} \aeq \min\{s,1-s\}$ and
\[
\int_{\R^N} \laps{s} f \laps{s} g= c_{N,s} \int_{\R^N} \int_{\R^N} \frac{(2f(x)-f(x+z)-f(x-z))g(x)}{|z|^{N+2s}}\, dz dx
\]
Now since $f, g \in \Sw(\R^N)$ and $2s \in (0,2)$, we may apply Fubini's theorem to interchange the $z$ and the $x$ integral, and use a similar change of variable as in \eqref{eq:shift}. Then 
\[
\int_{\R^N} \laps{s} f \laps{s} g
=c_{N,s} \int_{\R^N} \int_{\R^N} \frac{(f(x+z)-f(x))(g(x+z)-g(x))}{|z|^{N+2s}}\, dx dz.
\]
Hence using the bound for $c_{N,s}$, and writing $N+2s = \frac{N}{p_1}+t_1 + \frac{N}{p_2} + t_2$, we obtain
\[
\begin{split}
 [f]_{\dot{F}^{s}_{p,p}} \aleq
  \min\{s,(1-s)\} \abs{\int_{\R^N} \int_{\R^N} \frac{(f(x+z)-f(x))}{|z|^{\frac{N}{p_1}+t_1}}\frac{( g(x+z)-g(x))}{|z|^{\frac{N}{p_2}+t_2}}\, dx dz}.
 \end{split}
\]
Applying twice the integral H\"older inequality yields
\[
 [f]_{\dot{F}^{s}_{p,p}(\R^N)} \aleq \min\{s,(1-s)\} [f]_{{\dot{W}}^{t_1,p_1}(\R^N)}\, [g]_{{\dot{W}}^{t_2,p_2}(\R^N)} .
\]
This concludes the proof of the first inequality.

The proof of the second inequality is very similar. Instead of \Cref{th:dual} we use \Cref{th:dualinhom} and obtain instead of \eqref{eq:dualfl:34}
\[
[f]_{\dot{F}^{s}_{{p,q}}} \aleq [f]_{\dot{F}^{0}_{{p,q}}(\R^N)} + \abs{\int_{\R^N} \laps{s} f(x) \, \laps{s} g(x) \ dx },
\]
where this time we have $g \in \Sw(\R^N)$, $\mathcal{F}g$ supported on $\{|\xi| \geq 1/4\}$, and $[g]_{\dot{F}^{s}_{p',q'}(\R^N)} \leq 1$. The remaining arguments are the same.
\end{proof}

With \Cref{pr:dualityforlower} we obtain the lower bound of \eqref{eq:Fspppgeq2}.
\begin{proposition}
Let $p \in [2,\infty)$, $s \in (0,1)$ and $f \in \Sw(\R^N)$. Then 
\[
 \brac{\frac{1}{s^{\frac{1}{p}}} + \frac{1}{(1-s)^{\frac{1}{p}}}}\, [f]_{\dot{F}^{s}_{p,p}(\R^N)} \aleq [f]_{{\dot{W}}^{s,p}(\R^N)}.
\]
\end{proposition}
\begin{proof}
Since $p \in [2, \infty)$, we have $p' =\frac{p}{p-1} \in (1,2]$. So from the upper bound \eqref{eq:Fspppleq2} for $g \in \Sw(\R^N)$ we have
\[
 [g]_{{\dot{W}}^{s,p'}(\R^N)} \aleq \brac{\frac{1}{s^{\frac{1}{p'}}} + \frac{1}{(1-s)^{\frac{1}{p'}}}} [g]_{\dot{F}^{s}_{{p'},p'}(\R^N)}.
\]
From \Cref{pr:dualityforlower}, we thus obtain
\[
\begin{split}
 [f]_{\dot{F}^{s}_{p,p}(\R^N)}  
 \aleq&\min\{s,(1-s)\} \brac{\frac{1}{s^{\frac{1}{p'}}} + \frac{1}{(1-s)^{\frac{1}{p'}}}}
 [f]_{{\dot{W}}^{s,p}(\R^N)}.
 \end{split}
\]
Now we can conclude since for any $s \in (0,1)$
\begin{equation} \label{eq:s(1-s)p}
 \brac{\min\{s,(1-s)\} \brac{\frac{1}{s^{\frac{1}{p'}}} + \frac{1}{(1-s)^{\frac{1}{p'}}}}}^{-1} \aeq \min\{s,(1-s)\}^{-\frac{1}{p}} \aeq \brac{\frac{1}{s^{\frac{1}{p}}} + \frac{1}{(1-s)^{\frac{1}{p}}}}
\end{equation}
\end{proof}

We also obtain the lower bound of \eqref{eq:Fspppleq2}:
\begin{proposition} 
Let $p \in (1,2]$, $s \in (0,1)$ and $f \in \Sw(\R^N)$. Then 
\[
 \brac{\frac{1}{s^{\frac{1}{2}}} + \frac{1}{(1-s)^{\frac{1}{2}}}}\, [f]_{\dot{F}^{s}_{p,p}(\R^N)} \leq C [f]_{{\dot{W}}^{s,p}(\R^N)} \\
\]
\end{proposition}
\begin{proof}
Since $p \in (1,2]$, we have $p'=\frac{p}{p-1} \in [2,\infty)$. So from the upper bound of \eqref{eq:Fspppgeq2} for $g \in \Sw(\R^N)$ we have
\[
 [g]_{{\dot{W}}^{s,p'}(\R^N)} \aleq \brac{\frac{1}{s^{\frac{1}{2}}} + \frac{1}{(1-s)^{\frac{1}{2}}}}\, [g]_{\dot{F}^{s}_{{p'},p'}(\R^N)}.
\]
From \Cref{pr:dualityforlower}, we thus obtain
\[
\begin{split}
 [f]_{\dot{F}^{s}_{p,p}(\R^N)}  \aleq&\min\{s,(1-s)\} \brac{\frac{1}{s^{\frac{1}{2}}} + \frac{1}{(1-s)^{\frac{1}{2}}}} [f]_{{\dot{W}}^{s,p}(\R^N)}.
 \end{split}
\]
Now we can conclude since for any $s \in (0,1)$
\[
\brac{\min\{s,(1-s)\}  \brac{\frac{1}{s^{\frac{1}{2}}} + \frac{1}{(1-s)^{\frac{1}{2}}}}}^{-1} \aeq \min\{s,(1-s)\}^{-\frac{1}{2}} \aeq \brac{\frac{1}{s^{\frac{1}{2}}} + \frac{1}{(1-s)^{\frac{1}{2}}}}.
\]
\end{proof}

Next is the proof of \eqref{eq:mainFsp2:1}.
\begin{proposition}
Let $p \in (1,2]$, $s \in (0,1)$ and $f \in \Sw(\R^N)$. Then 
\[
 \brac{\frac{1}{s^{\frac{1}{p}}} + \frac{1}{(1-s)^{\frac{1}{p}}}}\, [f]_{\dot{F}^{s}_{p,2}(\R^N)} \aleq [f]_{{\dot{W}}^{s,p}(\R^N)}.
\]
\end{proposition}
\begin{proof}
 Since $p \in (1,2]$, we have $p' \in [2,\infty)$. So from the upper bound \eqref{eq:mainFsp2:2} for $g \in \Sw(\R^N)$ we have
\[
 [g]_{{\dot{W}}^{s,p'}(\R^N)} \aleq \brac{\frac{1}{s^{\frac{1}{p'}}} + \frac{1}{(1-s)^{\frac{1}{p'}}}} [g]_{\dot{F}^{s}_{{p',}2}(\R^N)}.
\]
From \Cref{pr:dualityforlower}, we thus obtain
\[
\begin{split}
 [f]_{\dot{F}^{s}_{p,2}(\R^N)}  \aleq& \min\{s,(1-s)\} \brac{\frac{1}{s^{\frac{1}{p'}}} + \frac{1}{(1-s)^{\frac{1}{p'}}}} [f]_{{\dot{W}}^{s,p}(\R^N)}\\
 \end{split}
\]
We conclude by using \eqref{eq:s(1-s)p}.
\end{proof}

The lower bound of \Cref{th:sobolev} is contained in the following statement:
\begin{proposition}
Let $\Lambda > 1$ such that $(1-s) \leq \frac{1}{2\Lambda}$. Let $\bar{r} \in (0,s)$ such that $(1-\bar{r})=\Lambda (1-s)$. Let $r \in [0,\bar{r}]$. Then 
\[
 [f]_{\dot{F}^{r}_{p,2}(\R^N)} \aleq \brac{\|f\|_{L^p(\R^N)} + (1-s)^{\frac{1}{p}} [f]_{{\dot{W}}^{s,p}(\R^N)}}.
\]
\end{proposition}
\begin{proof}
Since for $r \in [0, \bar{r}]$,
\[
 [f]_{\dot{F}^{r}_{p,2}} \leq [f]_{\dot F^{0}_{p,2}} + [f]_{\dot{F}^{\bar{r}}_{p,2}} \aleq \|f\|_{L^p} + [f]_{\dot{F}^{\bar{r}}_{p,2}}
\]
it suffices to prove the proposition when $r = \bar{r}$.
From \Cref{pr:dualityforlower} we have 
\[
 [f]_{\dot{F}^{\bar{r}}_{{p,2}}(\R^N)} \aleq \|f\|_{L^p(\R^N)} + \min\{\bar{r},(1-\bar{r})\} [f]_{{\dot{W}}^{s,p}(\R^N)} \sup_{\substack{[g]_{\dot{F}^{\bar{r}}_{{p'},2}} \leq 1\\ \text{supp} \mathcal{F}g \subset \{|\xi| \geq 1/4\}}} [g]_{{\dot{W}}^{2\bar{r}-s,p'}(\R^N)}.
\]
Now from \eqref{eq:sobolev1}, since $\frac{1}{2} s < \bar{r} < s$, which implies $0 < 2\bar{r}-s < \bar{r}$, we find for any $g \in \Sw(\R^N)$ with $\mathcal{F}g$ supported in $\{|\xi| \geq 1/4\}$,
\[
\begin{split}
 [g]_{{\dot{W}}^{2\bar{r}-s,p'}(\R^N)}  \aleq &\frac{1}{(2\bar{r}-s)^{\frac{1}{p}}}\, [g]_{\dot{F}^{0}_{{p',}2}(\R^N)} + \frac{1}{(s-\bar{r})^{\frac{1}{p}}}\, [g]_{\dot{F}^{\bar{r}}_{{p',}2}(\R^N)} \\
 \leq &\left( \frac{1}{(2\bar{r}-s)^{\frac{1}{p'}}} + \frac{1}{(s-\bar{r})^{\frac{1}{p'}}} \right)  [g]_{\dot{F}^{\bar{r}}_{{p',}2}(\R^N)}.
 \end{split}
\]
Here the support condition on $\mathcal{F}g$ guarantees that $[g]_{\dot{F}^{0}_{{p',}2}} \aleq [g]_{\dot{F}^{\bar{r}}_{{p',}2}}$.
So we arrive at 
\[
 [f]_{\dot{F}^{\bar{r}}_{{p,}2}(\R^N)} \aleq \|f\|_{L^p(\R^N)} + \min\{\bar{r},(1-\bar{r})\} \brac{\frac{1}{(2\bar{r}-s)^{\frac{1}{p'}}} + \frac{1}{(s-\bar{r})^{\frac{1}{p'}}}} [f]_{{\dot{W}}^{s,p}(\R^N)}
\]
Now we have 
\[
 \min\{\bar{r},(1-\bar{r})\} \aleq (1-s), \quad
 s-\bar{r} = (\Lambda -1) (1-s) \quad \text{and}
\quad 
 2\bar{r} -s = 2-2\Lambda (1-s)-s \geq 1-s.
 \]
So 
\[
 \min\{\bar{r},(1-\bar{r})\} \brac{\frac{1}{(2\bar{r}-s)^{\frac{1}{p'}}} + \frac{1}{(s-\bar{r})^{\frac{1}{p'}}}} \aleq \frac{1-s}{(1-s)^{\frac{1}{p'}}} = (1-s)^{\frac{1}{p}}.
\]
This establishes the claim of the proposition.
\end{proof}

\section{Strong Convergence as \texorpdfstring{$s \to 1$}{s to 1}: Proof of Corollary~\ref{co:strongconv}}\label{s:co:strongconv}
\begin{proof}[Proof of \Cref{co:strongconv}]
Let $p \in (1,\infty)$, assume that $f_k \in \Sw(\R^N)$ such that 
\[
 f_k \rightharpoonup f \quad \text{weakly in $L^p(\R^N)$ as $k \to \infty$}.
\]
Let $(s_k)_{k \in \N} \subset (0,1)$ such that $s_k \uparrow 1$ and assume that
\begin{equation}\tag{\ref{eq:co:strongconv435}}
\Lambda := \sup_{k} \brac{\|f_k\|_{L^p(\R^N)} + (1-s_k)^{\frac{1}{p}} [f_k]_{\dot W^{s_k,p}(\R^N)}} < \infty.
\end{equation}
First we claim that 
\begin{equation}\label{eq:bbm2:conv:1}
 \limsup_{k \to \infty} \|f_k\|_{L^p(\R^N)} + [f_k]_{\dot{F}^{r}_{p,2}(\R^N)} \aleq \Lambda \quad \forall r \in (0,1).
\end{equation}
with constant independent of $r$. If $p \leq 2$ this follows easily from \eqref{eq:mainFsp2:1}, but the following proof, using \eqref{eq:sobolev2} instead, works for all $p \in (1, \infty)$. Up to removing finitely many sequence elements, we may assume that $(1-s_k) < \frac{1}{4}$ for all $k \in \N$. From \eqref{eq:sobolev2} we have for any $r<1-2(1-s_k)$,
\[
 [f_k]_{\dot{F}^{r}_{p,2}(\R^N)} \leq C  \brac{\|f_k\|_{L^p(\R^N)} + (1-s_k)^{\frac{1}{p}} [f_k]_{{\dot{W}}^{s_k,p}(\R^N)}} \leq C\, \Lambda.
\]
Since $s_k \xrightarrow{k \to \infty} 1$, this proves \eqref{eq:bbm2:conv:1}.

In view of \Cref{la:triebelconv} we deduce from \eqref{eq:bbm2:conv:1} that $f \in L^p(\R^N)$ and $[f]_{\dot{F}^{1}_{{p,}2}(\R^N)} < \infty$ with
\[
 \|f\|_{L^p(\R^N)} + [f]_{\dot{F}^{1}_{{p,}2}(\R^N)} \aleq \Lambda.
\]
In view of {\Cref{la:LittlewoodPaley},}
we conclude that $f \in H^{1,p}(\R^N)$ and
\[
 \|f\|_{L^p(\R^N)} + \|\nabla f\|_{L^p(\R^N)} \aleq \Lambda.
\]
The locally strong convergence of $f_k \to f$ in $H^{t,p}$ for any $t \in (0,1)$ follows from Rellich's Theorem. 
More precisely, fix $0<t < r < 1$ and a ball $B(0,R)$ for some $R > 0$. 
Denote by $\eta \in C_c^\infty(B(0,2R))$, $\eta \equiv 1$ in $B(0,R)$ any usual bump function. We then have by \eqref{eq:bbm2:conv:1}
\[
 \sup_{k \in \N}  \|\eta f_k \|_{L^p(\R^N)} + \|\laps{r} (\eta f_k) \|_{L^p(\R^N)} \aleq \sup_{k \in \N}  \|f_k\|_{L^p(\R^N)} + \|\laps{r} f_k \|_{L^p(\R^N)} \aleq \Lambda < \infty.
\]
Here we have used the Coifman--McIntosh--Meyer commutator estimate for
\[
 [\eta ,\laps{\tilde{t}}] (g) := \eta \laps{\tilde{t}} g - \laps{\tilde{t}} (\eta g),
\]
which implies that for any $\tilde{t} \in (0,1)$
\[
 \|[\eta ,\laps{\tilde{t}}] (g)\|_{L^p(\R^N)} \aleq \brac{\|\eta\|_{L^\infty} + [\eta]_{\lip}}\, \|g\|_{L^p(\R^N)}.
\]
For an overview of these commutator estimates see, e.g., \cite{LS20}.

Then, $\eta f_k$ has uniformly compact support and is uniformly bounded in $H^{r,p}(\R^N)$ and thus, up to taking a subsequence, converges strongly in $H^{t,p}(\R^N)$ (this can be either proven via the usual Rellich--Kondrachov argument, or by interpolation theory). That is, after passing to a subsequence (which we denote by $f_{n_k}$)
\begin{equation}\label{eq:bbm2:conv:2253}
 \lim_{k \to \infty} \|\eta f_{n_k} -\eta f\|_{L^p(\R^N)}  + \|\laps{t} (\eta f_{n_k} -\eta f)\|_{L^p(\R^N)} =0.
\end{equation}
Repeating this argument for different balls $B(0,\Gamma R)$ (extracting subsequence again if necessary) we obtain that
\begin{equation}\label{eq:bbm2:conv:2254}
 \lim_{k \to \infty} \|f_{n_k}-f\|_{L^p(B(0,\Gamma R))} = 0 \quad \forall \Gamma > 0.
\end{equation}
Now we have 
\[
\begin{split}
 &\|\laps{t} f_{n_k}-\laps{t} f\|_{L^p(B(0,R))}\\
 \leq &\|\eta\laps{t} (f_{n_k}-f)\|_{L^p(B(0,R))} \\
 \leq& \|\laps{t} (\eta f_{n_k}- \eta f)\|_{L^p(B(0,R))} +  \|[\eta,\laps{t}] (f_{n_k}- f)\|_{L^p(B(0,R))}\\
 \leq& \|\laps{t} (\eta f_{n_k}- \eta f)\|_{L^p(B(0,R))} +  \|[\eta,\laps{t}] (\chi_{B(0,\Gamma R)}f_{n_k}- \chi_{B(0,\Gamma R)}f)\|_{L^p(B(0,R))} \\
 &+ \left \|[\eta,\laps{t}] (\chi_{B(0,\Gamma R)^c}f_{n_k}- \chi_{B(0,\Gamma R)^c}f) \right \|_{L^p(B(0,R))}.
 \end{split}
\]
By \eqref{eq:bbm2:conv:2253} we have 
\[
 \lim_{k \to \infty} \|\laps{t} (\eta f_{n_k}- \eta f)\|_{L^p(B(0,R))} = 0.
\]
By the Coifman--McIntosh--Meyer estimate and then \eqref{eq:bbm2:conv:2254} we have 
\[
 \lim_{k \to \infty}\|[\eta,\laps{t}] (\chi_{B(0,\Gamma R)}f_{n_k}- \chi_{B(0,\Gamma R)}f)\|_{L^p(B(0,R))} \leq C(\eta) \lim_{k \to \infty} \|f_{n_k}-f\|_{L^p(B(0,\Gamma R)} = 0.
\]
Lastly, observe that since $\eta \chi_{B(0,\Gamma R)^c} \equiv 0$,
\[
\left \|[\eta,\laps{t}] (\chi_{B(0,\Gamma R)^c}f_{n_k}- \chi_{B(0,\Gamma R)^c}f) \right \|_{L^p(B(0,R))}\leq \left \|\laps{t} \brac{\chi_{B(0,\Gamma R)^c}(f_{n_k}-f)} \right \|_{L^p(B(0,R))}.
\]
for $t \in (0,1)$, $\Gamma > 2$, and $x \in B(0,R)$ from the integral representation of the fractional Laplacian $\laps{t}$ we find
\[
 |\laps{t} \brac{\chi_{B(0,\Gamma R)^c}(f_{n_k}-f)}(x)| \leq  C(t) \int_{B(0,\Gamma R)^c} \frac{|f_{n_k}(y)-f(y)|}{|x-y|^{N+t}}\, dy \aleq (\Gamma R)^{-t-\frac{N}{p}} \|f_{n_k}-f\|_{L^p(\R^N)}
\]
Consequently, for any $\Gamma > 2$,
\[
 \lim_{k \to \infty} \left \|[\eta,\laps{t}] (\chi_{B(0,\Gamma R)^c}f_{n_k}- \chi_{B(0,\Gamma R)^c}f) \right \|_{L^p(B(0,R))} \aleq R^{-t} \Gamma^{-t-\frac{N}{p}} \Lambda.
\]
We conclude that for any $\Gamma > 2$,
\[
\begin{split}
 \lim_{k \to \infty} \|\laps{t} f_{n_k}-\laps{t} f\|_{L^p(B(0,R))}\aleq R^{-t} \Gamma^{-t-\frac{N}{p}} \Lambda.
 \end{split}
\]
Taking $\Gamma \to \infty$ we conclude
\[
 \lim_{k \to \infty} \|\laps{t} f_{n_k}-\laps{t} f\|_{L^p(B(0,R))} = 0.
\]
This holds for any $R > 0$ and thus in particular for any compact set $K \subset \R^N$
\[
 \lim_{k \to \infty} \|\laps{t} f_{n_k}-\laps{t} f\|_{L^p(K)} = 0.
\]
Since the weak limit $f$ is unique, we can apply this argument to any subsequence of $(f_k)_{k \in \N}$, and find that actually
\[
  \lim_{k \to \infty} \|\laps{t} f_k-\laps{t} f\|_{L^p(K)} = 0.
\]
This implies \eqref{eq:co:strongconv4386}. 
As for \eqref{eq:co:strongconv43862}, from Sobolev embedding one finds that for any $0 < \tilde{t} < t$ and $\tilde{K} \subset K$ both compact with $\dist(\tilde{K},\partial K) > 0$ we have 
\[
 [f_k-f]_{{\dot{W}}^{\tilde{t},p}(\tilde{K})} \leq C(t,\tilde{t},p,K,\tilde{K},N) \brac{ \|\laps{t} (f_k- f)\|_{L^p(K)} + \|f_k-f\|_{L^p(K)}}.
\]
So we conclude \eqref{eq:co:strongconv43862} as well.
\end{proof}

\appendix 
\section{Proof of {the BBM formula on \texorpdfstring{$\R^N$}{RN}}
}\label{s:bbmglobaleq}
For the convenience of the reader we give here the proof of the following {BBM formula on $\R^N$:

\begin{theorem}
For $1 < p < \infty$ and $f \in L^p(\R^N)$, one has 
\begin{equation} \label{eq:BBMRN}
 \|\nabla f\|_{L^p(\R^N)}  = \left( \frac{p}{k(p,N)} \right)^{1/p} \lim_{s \to 1^-} (1-s)^{\frac{1}{p}} [f]_{\dot W^{s,p}(\R^N)}
\end{equation}
in the sense that the left hand side of the equality is finite if and only if the right hand side is finite, in which case the two sides are equal. In fact, we have $f \in H^{1,p}(\R^N)$ as long as $\liminf_{s \to 1^-} (1-s)^{\frac{1}{p}} [f]_{\dot{W}^{s,p}(\R^N)} < \infty$.
\end{theorem}
}
\begin{proof}
{\textbf{Step 1.} First, we establish \eqref{eq:BBMRN} for $f \in C^1 \cap H^{1,p}(\R^N)$.} Let $R \geq 100$ and $s \in [\frac{1}{2},1)$. Then
\[
\begin{split}
 \abs{(1-s)^{\frac{1}{p}} [f]_{{\dot{W}}^{s,p}(\R^N)} - (1-s)^{\frac{1}{p}} [f]_{{\dot{W}}^{s,p}(B(R))}} \aleq& (1-s)^{\frac{1}{p}}\brac{ \int_{\R^N \backslash B(R)}\int_{\R^N}  \frac{|f(x)-f(y)|^p}{|x-y|^{N+sp}}\, dx\, dy}^{\frac{1}{p}}\\
\aleq&(1-s)^{\frac{1}{p}}\brac{ \int_{\R^N \backslash B(R)} \int_{|x-y| \leq \frac{1}{4}R} \frac{|f(x)-f(y)|^p}{|x-y|^{N+sp}}\, dx\, dy}^{\frac{1}{p}}\\
&+(1-s)^{\frac{1}{p}}\brac{ \int_{\R^N \backslash B(R)} \int_{|x-y| \geq \frac{1}{4} R} \frac{|f(x)|^p+|f(y)|^p}{|x-y|^{N+sp}}\, dx\, dy}^{\frac{1}{p}}.
 \end{split}
\]
We observe for the second term
\[
 (1-s)^{\frac{1}{p}}\brac{ \int_{\R^N \backslash B(R)} \int_{|x-y| \geq \frac{1}{4}R} \frac{|f(x)|^p+ |f(y)|^p}{|x-y|^{N+sp}}\, dx\, dy}^{\frac{1}{p}} \aleq \brac{\frac{(1-s)}{s}}^{\frac{1}{p}}\,  R^{-s} \|f\|_{L^p(\R^N)}.
\]
For the first term, we use
\[|f(x)-f(y)| \aleq |x-y| \brac{\mathcal{M}_{2|x-y|}|\nabla f(x)| + \mathcal{M}_{2|x-y|}|\nabla f(y)|},\] 
where
\[
 \mathcal{M}_{r} g(x) := \sup_{\sigma \in (0,r)} \mvint_{B(x,\sigma)} |g(z)|\, dz
\]
is the centered maximal function, cf. \cite{BH93,H96}. Then we have 
\[
\begin{split}
 &(1-s)^{\frac{1}{p}}\brac{ \int_{\R^N \backslash B(R)} \int_{|x-y| \leq \frac{1}{4} R} \frac{|f(x)-f(y)|^p}{|x-y|^{N+sp}}\, dx\, dy}^{\frac{1}{p}}\\
 \aleq&(1-s)^{\frac{1}{p}}\brac{ \int_{\R^N \backslash B(3R/4)}  \brac{\mathcal{M}_{R/2} |\nabla f(z)|}^p dz\, \int_{|w| \leq R} \frac{1}{|w|^{N+(s-1)p}}\, dw}^{\frac{1}{p}}\\
 \aleq& R^{1-s} \brac{ \int_{\R^N \backslash B(R/4)}  |\nabla f(z)|^p dz}^{\frac{1}{p}}
 \end{split}
\]
In the last step we used the maximal theorem. That is, we have shown that for any $s \in [\frac{1}{2},1)$
\begin{equation}\label{eq:bbmlocal:23}
 \abs{(1-s)^{\frac{1}{p}} [f]_{{\dot{W}}^{s,p}(\R^N)} - (1-s)^{\frac{1}{p}} [f]_{{\dot{W}}^{s,p}(B(R))}} \aleq R^{1-s} \|\nabla f\|_{L^p(\R^N \setminus B(R/4))} + \brac{1-s}^{\frac{1}{p}}\,  R^{-s} \|f\|_{L^p(\R^N)}.
\end{equation}
Now we can conclude from the local case in \cite{BBM01}; recall that in \cite[Corollary 2]{BBM01} it is proven that for any $R > 0$ 
\begin{equation}\label{eq:bbmlocal}
  \|\nabla f\|_{L^p(B(R))}  = \left( \frac{p}{k(p,N)} \right)^{1/p} \lim_{s \to 1^-} (1-s)^{\frac{1}{p}} [f]_{{\dot{W}}^{s,p}(B(R))},
\end{equation}
where $k(p,N) := \int_{\S^{N-1}} |e \cdot \omega|^p d\omega$ and $e$ is any unit vector in $\R^N$. Fix $\eps > 0$. Since $\nabla f \in L^p(\R^N)$ there must be a large radius $R>0$ such that 
\begin{equation}\label{eq:bbmlocal:24}
 \|\nabla f\|_{L^p(\R^N \backslash B(R/4))} < \eps.
\end{equation}
Then
\begin{align*}
 &\abs{\left( \frac{p}{k(p,N)} \right)^{1/p} (1-s)^{\frac{1}{p}} [f]_{{\dot{W}}^{s,p}(\R^N)} - \|\nabla f\|_{L^p(\R^N)}} \\
\leq & \,  \eps + \abs{\left( \frac{p}{k(p,N)} \right)^{1/p} (1-s)^{\frac{1}{p}} [f]_{{\dot{W}}^{s,p}(\R^N)} - \|\nabla f\|_{L^p(B(R))}} \\
\aleq &  \, \eps + R^{1-s} \|\nabla f\|_{L^p(\R^N \setminus B(R/4))} + \brac{1-s}^{\frac{1}{p}}\,  R^{-s} \|f\|_{L^p(\R^N)} + \abs{\left( \frac{p}{k(p,N)} \right)^{1/p} (1-s)^{\frac{1}{p}} [f]_{{\dot{W}}^{s,p}(B(R))} - \|\nabla f\|_{L^p(B(R))}} \\
\leq &  \, (R^{1-s} + 1) \eps + \brac{1-s}^{\frac{1}{p}}\,  R^{-s} \|f\|_{L^p(\R^N)} + \abs{\left( \frac{p}{k(p,N)} \right)^{1/p} (1-s)^{\frac{1}{p}} [f]_{{\dot{W}}^{s,p}(B(R))} - \|\nabla f\|_{L^p(B(R))}}
\end{align*}
where the first and the third inequality follows from \eqref{eq:bbmlocal:24} and the second inequality follows from  \eqref{eq:bbmlocal:23}. Since $R$ is fixed once $\eps$ is fixed, we may let $s \to 1^-$ and use \eqref{eq:bbmlocal}. This shows
\[
\limsup_{s \to 1^-} \abs{\left( \frac{p}{k(p,N)} \right)^{1/p} (1-s)^{\frac{1}{p}} [f]_{{\dot{W}}^{s,p}(\R^N)} - \|\nabla f\|_{L^p(\R^N)}} \aleq \eps,
\]
but since $\eps > 0$ is arbitrary, this proves
\[
\left( \frac{p}{k(p,N)} \right)^{1/p} \lim_{s \to 1^-} (1-s)^{\frac{1}{p}} [f]_{{\dot{W}}^{s,p}(\R^N)} = \|\nabla f\|_{L^p(\R^N)}.
\]

{\textbf{Step 2.} Next, assume $f \in H^{1,p}(\R^N)$. We show that
\begin{equation} \label{eq:liminf_upperbdd}
\limsup_{s \to 1^-} (1-s)^{1/p} [f]_{\dot{W}^{s,p}(\R^N)} \leq \left( \frac{k(p,N)}{p} \right)^{1/p} \|\nabla f\|_{L^p(\R^N)}.
\end{equation}
We first observe that for $f \in H^{1,p}(\R^N)$ and $s \in [1/2,1)$,
\[
\begin{split}
[f]_{\dot{W}^{s,p}(\R^N)} &= \left( \int_{\R^N} \int_{\R^N} \frac{|f(x+h)-f(x)|^p}{|h|^{N+sp}} dx dh \right)^{1/p} \\
&\leq \left( \int_{|h| \leq 1} \int_{\R^N} \frac{|f(x+h)-f(x)|^p}{|h|^p} dx \frac{|h|^{(1-s)p}}{|h|^{N}} dh \right)^{1/p} \\
& \qquad + \left( \int_{|h| > 1} \int_{\R^N} (|f(x+h)|^p + |f(x)|^p) dx \frac{1}{|h|^{N+sp}} dh \right)^{1/p}.
\end{split}
\]
We appeal to the facts that 
\[
\int_{\R^N} |f(x+h)-f(x)|^p dx \leq |h|^p \|\nabla f\|_{L^p(\R^N)}^p
\]
(see \cite[Proposition 9.3]{Brezis_book}) and that $\int_{\R^N} |f(x+h)|^p + |f(x)|^p dx = 2 \|f\|_{L^p(\R^N)}^p$. Thus
\[
\begin{split}
[f]_{\dot{W}^{s,p}(\R^N)} &\lesssim \left(\int_{|h| \leq 1} \frac{|h|^{(1-s)p}}{|h|^{N}}  dh \right)^{1/p} \|\nabla f\|_{L^p(\R^N)} + \left(\int_{|h| > 1} \frac{1}{|h|^{N+sp}} dh \right)^{1/p} \|f\|_{L^p(\R^N)} \\
&\lesssim \frac{1}{(1-s)^{1/p}} \|\nabla f\|_{L^p(\R^N)} + \frac{1}{s^{1/p}} \|f\|_{L^p(\R^N)}
\end{split}
\]
which implies
\begin{equation} \label{eq:generalsupest}
\sup_{1/2 \leq s < 1} (1-s)^{1/p} [f]_{\dot{W}^{s,p}(\R^N)} \leq C \|f\|_{H^{1,p}(\R^N)}
\end{equation}
for some constant $C = C_{p,N}$, whenever $f \in H^{1,p}(\R^N)$ (note that \eqref{eq:generalsupest} strengthens the second conclusion of \Cref{co:fracBBMupper} by weakening the hypothesis on $f$ from $f \in \Sw(\R^N)$ to $f \in H^{1,p}(\R^N)$).
To proceed further, for any $\varepsilon > 0$, pick $g \in C^1 \cap H^{1,p}(\R^N)$ so that $\|f-g\|_{H^{1,p}(\R^N)} < \varepsilon$. Then 
\[
\begin{split}
(1-s)^{1/p} [f]_{\dot{W}^{s,p}(\R^N)} & \leq (1-s)^{1/p}  [g]_{\dot{W}^{s,p}(\R^N)} +  (1-s)^{1/p} [f-g]_{\dot{W}^{s,p}(\R^N)} \\
&\leq (1-s)^{1/p} [g]_{\dot{W}^{s,p}(\R^N)} + C \varepsilon,
\end{split}
\]
where in the last inequality we applied \eqref{eq:generalsupest} to $f-g \in H^{1,p}(\R^N)$ in place of $f$.
Now recall \eqref{eq:BBMRN} has already been proved for $g \in C^1 \cap H^{1,p}(\R^N)$. As a result, letting $s \to 1^-$, we obtain
\[
\begin{split}
\limsup_{s \to 1^-} (1-s)^{1/p} [f]_{\dot{W}^{s,p}(\R^N)} & \leq  \left( \frac{k(p,N)}{p} \right)^{1/p} \|\nabla g\|_{L^p(\R^N)} + \varepsilon \\
&\leq \left( \frac{k(p,N)}{p} \right)^{1/p} (\|\nabla f\|_{L^p(\R^N)} + \varepsilon) + C \varepsilon.
\end{split}
\]
Since $\varepsilon > 0$ is arbitrary, \eqref{eq:liminf_upperbdd} follows.

\textbf{Step 3.} Finally, assume $f \in L^p(\R^N)$ and
\[
A := \liminf_{s \to 1^-} (1-s)^{1/p} [f]_{\dot{W}^{s,p}(\R^N)} < \infty. 
\]
It is known that then $f \in H^{1,p}(\R^N)$ and 
\begin{equation} \label{eq:liminf_lowerbdd}
\|\nabla f\|_{L^p(\R^N)} \leq \left(\frac{p}{k(p,N)}\right)^{1/p} A.
\end{equation}
In fact, then for every bounded smooth domain $\Omega \subset \R^N$, we have
\[
\liminf_{s \to 1^-} (1-s)^{1/p} [f]_{\dot{W}^{s,p}(\Omega)} \leq A < \infty,
\] 
so \cite[Theorem 2]{BBM01} (and its proof) shows that $f \in H^{1,p}(\Omega)$ with 
\[
\|\nabla f\|_{L^p(\Omega)} \leq \left(\frac{p}{k(p,N)}\right)^{1/p} A.
\]
Since $\Omega$ is an arbitrary bounded smooth domain in $\R^N$, this shows $f \in H^{1,p}(\R^N)$ and that \eqref{eq:liminf_lowerbdd} holds.
}
\end{proof}

\section{Proof of Corollary~\ref{co:sharpshob2}} \label{s:co:sharpshob2}
\begin{proof}[Proof of \Cref{co:sharpshob2}]
{Let $0 < s \leq t <1$ and $f \in \Sw(\R^N)$.} From \Cref{th:mainFspp} we have 
\[
 \min\{s,(1-s)\}^{\frac{1}{2}} [f]_{{\dot{W}}^{s,2}(\R^N)} \aleq [f]_{\dot{F}^{s}_{2,2}}
\]
and
\[
 [f]_{\dot{F}^{t}_{{2,}2}} \aleq \min\{t,(1-t)\}^{\frac{1}{2}} [f]_{{\dot{W}}^{t,2}(\R^N)}.
\]
The result now follows from the inequality $[f]_{\dot{F}^{s}_{2,2}} \aleq \|f\|_{L^2} + [f]_{\dot{F}^{t}_{2,2}}$ when $0 \leq s \leq t$.
\end{proof}

\bibliographystyle{abbrv}
\bibliography{bib}

\end{document}